\documentclass[a4paper, 11pt, reqno]{amsart}
\usepackage{amsmath,amsfonts,amssymb,graphics,mathrsfs,amsthm,eurosym,verbatim,enumerate,subcaption}
\usepackage[marginratio=1:1,tmargin=117pt, height=650pt]{geometry}
\usepackage[usenames, dvipsnames]{xcolor}
\usepackage[shortlabels]{enumitem}

\usepackage[normalem]{ulem}
\usepackage{bm}
\usepackage{tikz-cd}
\usepackage[colorlinks=true,linkcolor=blue!60!black,citecolor=teal!80!black,urlcolor=RoyalBlue,hypertexnames=false]{hyperref}

\numberwithin{equation}{section}
\usepackage{cleveref}
\makeatletter
\def\blfootnote{\xdef\@thefnmark{}\@footnotetext}
\makeatother
\theoremstyle{plain}
\newtheorem{theorem}{Theorem}[section]
\newtheorem{proposition}[theorem]{Proposition}
\newtheorem{corollary}[theorem]{Corollary}
\newtheorem{lemma}[theorem]{Lemma}
\newtheorem{definition}[theorem]{Definition}

\newcommand*{\defeq}{\mathrel{\vcenter{\baselineskip0.5ex \lineskiplimit0pt
 \hbox{\scriptsize.}\hbox{\scriptsize.}}}=}
\newcommand{\eqdef}{=\mathrel{\vcenter{\baselineskip0.5ex \lineskiplimit0pt
	\hbox{\scriptsize.}\hbox{\scriptsize.}}}}
\newtheorem*{remark}{Remark}
\newtheorem{observation}[theorem]{Observation}

\theoremstyle{remark}

\newcommand{\HH}{{\mathbb{H}}}
\newcommand{\C}{{\mathbb{C}}}

\newcommand{\R}{{\mathbb{R}}}
\newcommand{\Z}{{\mathbb{Z}}}
\newcommand{\N}{{\mathbb{N}}}

\newcommand{\M}{{\mathcal{M}}}
\newcommand{\Mlog}{{\mathcal{M}_{\log}}}
\newcommand{\B}{{\mathcal{B}}}

\newcommand{\qfor}{\quad\text{for }}
\newcommand{\uldelta}{\underline{\delta}}
\newcommand{\uleta}{\underline{\eta}}

\newcommand{\ulr}{\underline{r}}
\newcommand{\ult}{\underline{t}}

\newcommand{\Rea}{\operatorname{Re }}
\newcommand{\Ima}{\operatorname{Im }}

\newcommand{\dist}{\operatorname{dist}}

\begin{document}
\title[A finite-order answer to a problem of Erd\H{o}s]
{A finite-order answer to a problem of Erd\H{o}s on maximum modulus points}
\author[{L. Pardo-Sim\'on \and D. J. Sixsmith}]{Leticia Pardo-Sim\'on \and David J. Sixsmith}

\address{\noindent Dept. de Matemàtiques i Informàtica\\ Universitat de Barcelona\\ Spain\\
	\newline  Centre de Recerca Matemàtica\\ Bellaterra\\ Catalonia\\ Spain.
	\textsc{\newline \indent 
		\href{https://orcid.org/0000-0003-4039-5556%
		}{\includegraphics[width=1em,height=1em]{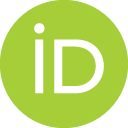} {\normalfont https://orcid.org/0000-0003-4039-5556}}
}}
\email{lpardosimon@ub.edu}

\address{School of Mathematics and Statistics\\ The Open University\\
Milton Keynes MK7 6AA\\ UK\textsc{\newline \indent \href{https://orcid.org/0000-0002-3543-6969}{\includegraphics[width=1em,height=1em]{orcid2.png} {\normalfont https://orcid.org/0000-0002-3543-6969}}}}
\email{david.sixsmith@open.ac.uk}

\thanks{This work was partially supported by the project PID2023-147252NB-I00 financed by funded by MICIU/AEI/10.13039/501100011033 and by FEDER, EU, and by the Spanish State Research Agency, through the Severo Ochoa and María de Maeztu Program for Centers and Units of Excellence in R\&D (CEX2020-001084-M). The first author is a Serra Húnter fellow.
}
\begin{abstract}
In 1964, Erd\H{o}s asked whether, for a non-monomial entire function, the number of maximum modulus points on the circle  \(|z|=r\) can become arbitrarily large as $r\to\infty$. In 1968, Herzog and Piranian answered this question affirmatively, but without quantitative control on the resulting function. We prove that such an example can be chosen to have finite order and, moreover, to belong to the Eremenko--Lyubich class \(\B\). We also prove an interpolation theorem for maximum modulus sets: prescribed points with pairwise distinct moduli can be forced to lie in the maximum modulus set of some function in \(\B\), with finite order under a geometric separation condition.
\end{abstract}
\maketitle

\section{Introduction}
For \(r\geq 0\), write \(M(r,f)\defeq \max_{|z|=r}|f(z)|\). Following \cite{Sixsmithmax}, we denote by \(\M(f)\) the \emph{maximum modulus set} of \(f\), that is,
\begin{equation}
	\label{Mdef}
	\M(f)\defeq \{z\in\C\colon |f(z)|=M(|z|,f)\}.
\end{equation}
If \(f\) is a monomial, then \(\M(f)=\C\). Otherwise, \(\M(f)\) consists of a union of closed maximum curves, which are analytic except possibly at their endpoints; see \cite[Theorem 10]{valironlectures} or \cite{Blumenthal}. Maximum modulus sets have been studied by many authors; see, for example, \cite{hardy1909, jassimlondon,tyler,Csordas_1990,Blu_conj, letidave_polynomial}.

For a non-monomial entire function \(f\), the intersection of \(\M(f)\) with each circle \(|z|=r\), \(r>0\), is finite. We denote its cardinality by
\[
v_f(r)\defeq \#(\M(f)\cap\{|z|=r\}).
\]
The elementary examples \(f_n(z)=\exp(z^n)\), $n\in \N$, show that there is no local obstruction to having many maximum modulus points on a circle: in this case \(v_{f_n}(r)=n\) for every \(r>0\). Erd\H{o}s asked whether this phenomenon can occur, with \(n\) tending to infinity, for a single non-monomial entire function. More precisely, he asked the following.

\noindent\textbf{Question \textup{(Erd\H{o}s, 1964).}}
Let \(f\) be a non-monomial entire function.
\begin{enumerate}
	\item Can \(v_f\) be unbounded, that is, can one have \(\limsup_{r\to\infty}v_f(r)=\infty\)?
	\item Can \(v_f\) tend to infinity, that is, can one have \(\liminf_{r\to\infty}v_f(r)=\infty\)?
\end{enumerate}

\noindent
The problem seems first to have appeared in Hayman's book on research problems in function theory \cite{Hayman1967}; see also \cite{leti_Adi_Erdos} for further historical discussion, including the later attribution of the second part to Clunie.

The first part of Erd\H{o}s' question was answered positively by Herzog and Piranian \cite{HerzogPiranian1968}, who constructed an entire function \(f\) such that \(v_f(n)=n\) for every \(n\in\N\). Their construction is ingenious, but very brief: it is presented as a sketch based on successive refinements, and it does not give quantitative estimates for the resulting function. In particular, it gives no control on the order of growth of \(f\), and no information about its singular values. It also gives no useful control of \(v_f(r)\) away from the prescribed radii, so it does not address the second part of the question.

The second part remains open. An approximate positive result was obtained in \cite{leti_Adi_Erdos}: there is an entire function for which, on every sufficiently large circle, many arcs are arbitrarily close to the maximum modulus and are separated by many arcs where the function is close to zero. However, the function constructed there has infinite lower order. On the other hand, Marchenko has conjectured that the answer to the second part of Erd\H{o}s' question should be negative for functions of finite lower order \cite{Marchenko2012}. Thus finite order is not a harmless technical restriction in this problem; it is a natural boundary at which even the weaker, \(\limsup\), phenomenon becomes interesting.

Recall that the order of an entire function \(f\) is
\[
\rho(f)\defeq \limsup_{r\to\infty}\frac{\log\log M(r,f)}{\log r}.
\]
Recall also that the Eremenko--Lyubich class \(\B\) consists of transcendental entire functions whose set of finite singular values is bounded. Finite order gives global control on the growth of \(f\), while being in class \(\B\) gives control on the singularities of its inverse branches. It is therefore natural to ask whether the Herzog--Piranian phenomenon is compatible with these two standard restrictions.

Our first main result shows that it is.

\begin{theorem}
	\label{th:1}
	There is an entire function \(f\in\B\), of finite order, such that
	\[
	\limsup_{r\to\infty} v_f(r)=\infty.
	\]
\end{theorem}

The proof of Theorem~\ref{th:1} uses logarithmic coordinates. \textit{Lifting} by the exponential map, circles become vertical lines, and maximum modulus points become points where the associated logarithmic transform is maximal on a vertical line. We construct a conformal model whose tract has, at selected real parts, many competing channels. The geometry is arranged so that many channels realise the maximum simultaneously. We then approximate the model by an entire function in class \(\B\), using the approximation theorem of Rempe \cite{lasse_model14}. The main difficulty is to keep the tract sufficiently regular to obtain finite order.

The same method also gives a different kind of flexibility. Instead of asking how many maximum modulus points can lie on one circle, one may ask how freely the position of maximum modulus points can be prescribed. Our second result says that, within the class \(\B\), essentially any escaping sequence of points on distinct circles can be forced to lie in the maximum modulus set.

\begin{theorem}
	\label{th:2}
	Suppose that \((z_n)_{n\in\N}\) is a sequence of points in the plane, tending to infinity, and such that no two points have the same modulus. There is an entire function \(f\in\B\) such that \(z_n\in\M(f)\), for all \(n\in\N\). Moreover, if there is a constant \(K>1\) such that \(|z_{n+1}|\geq K|z_n|\), for all \(n\in\N\), then \(f\) can be taken to have finite order.
\end{theorem}

\begin{remark}\normalfont
	The escape condition in Theorem~\ref{th:2} is natural. Indeed, near the origin, the maximum modulus set of a non-monomial entire function consists of finitely many analytic arcs meeting with equal angles; see \cite{Hayman, leti_dave_vasso}. In a forthcoming paper \cite{junctionpoints}, we show that an analogous local description holds near any point of the plane. Hence arbitrary finite accumulation of prescribed maximum modulus points cannot be expected. For example, for any \(c\in\C\), the sequence \(z_n\defeq c+2^{-n}e^{in}\), \(n\in\N\), cannot be contained in the maximum modulus set of a non-monomial entire function.
\end{remark}
%
%
\subsection*{Notation}
For a set $S \subset \C$ and $a \in \C$ we define
$S+a \defeq \{ z + a : z \in S \}.$ If $z,w$ are two complex numbers, we use $[z,w]$ to denote the straight line segment from $z$ to $w$. If $V$ is a domain and $z, w \in V$, then we use $d_V(z, w)$ to denote the hyperbolic distance in $V$ between $z$ and $w$. We let $\mathbb{H}$ denote the right half-plane. 

%
%


\section{Background to the constructions}

In this section we outline the general framework used in the proofs of Theorems~\ref{th:1} and~\ref{th:2}. The starting point is the ``model/approximation'' approach developed by Rempe \cite{lasse_model14}, in which one first constructs a suitable conformal model in logarithmic coordinates and then approximates it by an entire function in the Eremenko--Lyubich class. 

\begin{definition}[{\cite[Definition 1.6]{lasse_model14}}]\label{def:model}\normalfont
	A \emph{model} is a triplet $(G, V, H)$ with the following properties.
	\begin{enumerate}[(i)]
		\item $H \subset \C$ is a simply-connected domain that contains $\mathbb{H}$.
		\item $V \subset \C$ is a simply-connected domain, disjoint from its $2\pi i$ translates.
		\item $G \colon V \rightarrow H$ is a conformal isomorphism, such that if $\{z_n\}_{n\in\N}\subset H$ is a sequence with $z_n \rightarrow \infty$ in $H$, then $\Rea G^{-1}(z_n)\rightarrow +\infty$.
	\end{enumerate}
\end{definition}

\begin{remark}\normalfont
	The equivalent definition in \cite{lasse_model14} defines a model function using a conformal map $\Psi \colon \exp(V) \to H$, such that $G = \Psi \circ \exp$. Our definition, and the theorem below, are equivalent, and slightly easier to work with in our setting.
\end{remark}

Following \cite{lasse_model14}, we make the following specific choice of $H$ in both of our constructions; this domain is fixed throughout the paper.
\begin{equation}
	\label{eq_H}
	H \defeq \{ x+iy \colon x> -14\log_+\vert y\vert \},
\end{equation}
where $\log_+(t)\defeq \max(0, \log t)$. For simplicity in calculations, we will also assume that $-1\in V$.

Since $V$ is disjoint from its non-zero $2 \pi i$ translates, the exponential map is injective on $V$, and hence there is a single-valued branch of the logarithm on $\exp(V)$ taking values in $V$. Using this logarithm, we can then define an analytic function $g \colon \exp(V) \to \exp(H)$ with the property that
\begin{equation}
	\label{eq:gG}
	g \circ \exp = \exp \circ G.
\end{equation} 

The following result, which is a merged version of \cite[Theorem 1.7]{lasse_model14} and \cite[Corollary 4.5]{lasse_model14}, and is based on the use of Cauchy integrals, allows us to approximate the function $g$ by an entire function $f \in \B$.
\begin{theorem}
	\label{thm_approx1.7}
	Let $(G, V, H)$ be a model, with $H$ fixed in \eqref{eq_H}, $-1\in V$, and assume that $G(-1)=1$.
	Let $g$ be as in \eqref{eq:gG}.
	Then there exists $f\in \mathcal{B}$ such that
	\begin{equation}
		\label{eq_f}
		f(z) \defeq 
		\begin{cases}
			g(z)+h(z), &\text{for } z\in \exp(V), \\
			h(z),   &\text{otherwise.}
		\end{cases}
	\end{equation}
	Here $h\colon \C \rightarrow \C$ is such that, for some constant $M>0$ universal under this normalization, $\vert h(z)\vert \leq M/\vert z \vert $ for $|z| > 1$.

Moreover, suppose that \(T\defeq \exp(V)\) is symmetric with respect to the real axis and that, writing \(\Psi\colon T\to H\) for the map satisfying
	\(G=\Psi\circ\exp\), we have \(\Psi(T\cap\R)\subset\R\). Then \(f\) may be
	chosen so that \(f(\R)\subset\R\). Equivalently, since \(f\) is entire,
	\(f(\bar z)=\overline{f(z)}\) for all \(z\in\C\).
\end{theorem}
For an entire function $f$, define
\[
\Mlog(f) \defeq \{ z \in \C \colon |f(e^z)| = \max_{\substack{ w \in \C \\ \Rea w = \Rea z}} |f(e^w)| \}.
\] 
Note that by $2\pi i$-periodicity of the exponential map, 
\begin{equation}\label{eq_max_translates}
	z\in \Mlog(f) \iff z + 2 k\pi i \in \Mlog(f) \text{ for } k\in \Z,
\end{equation}
and since the exponential map sends vertical lines to circles, we have $$\M(f)=\exp(\Mlog(f))\cup\{0\}.$$

In order to obtain an entire function $f$ such that $\M(f)$ has the exact properties we are looking for, we will define a family of models $(G_{_{\uldelta}}, V(\uldelta), H)$ that depend continuously on an infinite-dimensional parameter $\uldelta=(\delta_n)_{n\in\N}$ belonging to a space of the form 
\begin{equation} \label{eq_Delta}
	\Delta\defeq \prod_{n\in\N} [a_n,b_n], \qquad a_n,b_n \in \R,\quad a_n<b_n.
\end{equation}

Applying Theorem~\ref{thm_approx1.7} to each such model, we will obtain a family of entire maps~$f_{_{\uldelta}}$. 

For the families considered below, the phrase ``normalised conformal isomorphism'' means that \(G_{_{\uldelta}}(-1)=1\) and that \(G_{_{\uldelta}}\) sends the right-hand end of \(V(\uldelta)\) to the end of \(H\) at infinity; equivalently, \(G_{_{\uldelta}}(z)\to\infty\) as \(z\in V(\uldelta)\) tends to the end \(\Rea z\to+\infty\).
\begin{observation}\label{obs_continuous_dependence} \normalfont
Suppose that the domains $V(\uldelta)$, $\uldelta\in\Delta$, depend continuously on $\uldelta$ in the Carath\'eodory kernel topology, where $\Delta$ is endowed with the product topology, and let $G_{_{\uldelta}}\colon V(\uldelta)\to H$ denote the corresponding normalised conformal isomorphisms. Then, by the Carath\'eodory kernel theorem \cite[Theorem 1.8]{Pommerenke}, the functions $G_{_{\uldelta}}$ depend continuously on $\uldelta$ in the topology of locally uniform convergence, and in turn, so do the functions $g_{_{\uldelta}}$ from \eqref{eq:gG}. Moreover, it follows from the construction of the functions $f_{_{\uldelta}}$ in \cite{lasse_model14} that $f_{_{\uldelta}}$ depends continuously on $\uldelta$ in the topology of locally uniform convergence; see \cite[Theorem 2.1]{lasse_model14} and the proof of \cite[Corollary 4.5]{lasse_model14}. The same construction also gives locally uniform continuous dependence when the
models satisfy the symmetry hypotheses in Theorem~\ref{thm_approx1.7} and the
approximating functions are chosen real-symmetric.
\end{observation}
\begin{remark}[Carath\'eodory kernel convergence]\normalfont
We use Carath\'eodory kernel convergence in the following standard sense. If \(D_j\) are domains containing a fixed base point \(a\), the kernel of the sequence is the largest domain \(D\) containing \(a\) such that every compact subset of \(D\) is contained in \(D_j\) for all sufficiently large \(j\). We say that \(D_j\to D\) in the kernel sense if \(D\) is the kernel of every subsequence. Thus, in the applications below, after rescaling, the limiting domain is obtained by retaining exactly those boundary pieces that remain visible on compact sets. The Carath\'eodory kernel theorem says that, with the usual normalisation, the corresponding conformal maps converge locally uniformly, with convergence of derivatives.
\end{remark}
In order to show that there is a suitable choice of the variable $\uldelta$ that gives the result we are looking for, we will use the following auxiliary lemma.
\begin{lemma}\label{lem:shooting}
Let $\Delta$ be of the form \eqref{eq_Delta}, and write $\uldelta=(\delta_n)_{n\in\N}$. For each $n\in \N$, let $\phi_n\colon \Delta\to \R$ be continuous. Suppose that, for each \(n\in\N\), there is a sign $\sigma_n\in\{-1,1\}$ such that, for all $\uldelta\in\Delta$,
	\[\sigma_n\cdot\phi_n(\uldelta)\geq0\quad\text{when }\delta_n=a_n,\qquad
		\sigma_n\cdot\phi_n(\uldelta)\leq0\quad\text{when }\delta_n=b_n.
	\]
Then, for any $A\subseteq \N$, there exists $\uldelta \in \Delta$ such that $\phi_n(\uldelta)=0$ for all $n\in A$.
\end{lemma}

\begin{remark}\normalfont
		This shooting lemma is inspired by related arguments in
		\cite[Lemma~4.6]{letidave_variations} and
		\cite[Proof of Theorem~7.4]{lasse_model14}. 
		The proof given here is shorter: the finite-dimensional case is
		a direct application of the Poincar\'e--Miranda theorem, a
		multidimensional generalisation of the intermediate value theorem,
		and the general case follows by a compactness argument.
\end{remark}

\begin{proof}[Proof of Lemma \ref{lem:shooting}]
If $A=\emptyset$, there is nothing to prove. First suppose that $A$ is finite, say $A=\{\alpha_1,\ldots,\alpha_m\}$.
	Write $\Delta(A)\defeq \prod_{k=1}^m [a_{\alpha_k},b_{\alpha_k}]$ and, for
	$y=(y_1,\ldots,y_m)\in \Delta(A)$, define $\uldelta(y)\in\Delta$ by
	\[
	\delta_n \defeq
	\begin{cases}
		y_k, &\text{ if } n = \alpha_k \text{ for some } k,\\
		a_n, &\text{ otherwise.}
	\end{cases}
	\]
	Define a continuous map $F\colon \Delta(A)\to \R^m$ by
	\[
	F(y)\defeq \big(\sigma_{\alpha_1}\phi_{\alpha_1}(\uldelta(y)),\ldots,
		\sigma_{\alpha_m}\phi_{\alpha_m}(\uldelta(y))\big).
	\]
	For each $k\in\{1,\ldots,m\}$, on the face $\{y_k=a_{\alpha_k}\}$ we have $F_k(y)\geq0$, and on the opposite face $\{y_k=b_{\alpha_k}\}$ we have $F_k(y)\leq0$,
	by the hypotheses of the lemma. Hence $F$ satisfies the sign conditions of the
	Poincar\'e--Miranda theorem; therefore there exists $y^*\in\Delta(A)$ such that
	$F(y^*)=0$. Setting $\uldelta\defeq \uldelta(y^*)$ gives $\phi_n(\uldelta)=0$ for all
	$n\in A$.
	
	Now suppose that $A$ is infinite. Choose an increasing sequence of finite subsets
	$A_1\subset A_2\subset\cdots$ with $\bigcup_{k\ge1}A_k=A$. By the finite case, for each
	$k$ there exists $\uldelta^{(k)}\in\Delta$ such that $\phi_n(\uldelta^{(k)})=0$ for all
	$n\in A_k$. Since $\Delta$ is a countable product of compact intervals, it is compact and metrizable in the product topology, for instance by Tychonoff's theorem; hence
	there exists a subsequence $\uldelta^{(k_j)}$ converging to some $\uldelta\in\Delta$.
	Fix $n\in A$. Then $n\in A_{k_j}$ for all sufficiently large $j$, so
	$\phi_n(\uldelta^{(k_j)})=0$ for all large~$j$. By continuity of $\phi_n$ we obtain
	$\phi_n(\uldelta)=0$. Since $n\in A$ was arbitrary, this holds for all $n\in A$.
\end{proof}

Recall that for a hyperbolic domain $U \subset \C$, we denote by $\rho_U(z)$ the hyperbolic density at the point $z \in U$. If $U$ is simply-connected, then, by \cite[Theorems~8.2 and 8.6]{beardonandminda}, we have the following standard estimate:
\begin{equation}\label{eq:hypest}
	\frac{1}{2\operatorname{dist}(z, \partial U)} \leq \rho_U(z) \leq \frac{2}{\operatorname{dist}(z, \partial U)}, \qfor z \in U,
\end{equation}
where $\dist(z, \partial U)$ is the Euclidean distance between $z$ and $\partial U$.

Finally, in order to prove that the functions we construct are of finite order, we will use the following results. The next theorem is a version of Ahlfors' distortion theorem; see, for example, \cite[Corollary to Theorem 4.8]{ahlfors}. 

Here, and below, a maximal vertical line segment in a domain \(V\) at real part \(t\) means a connected component of \(V\cap\{z:\Rea z=t\}\).

\begin{theorem}\label{thm_ahlfors}
	Let $V \subset \mathbb{C}$ be a simply connected domain, and
	let $\gamma_t, \gamma_{t'} \subset V$ be two maximal vertical line segments at real parts $t<t'$ respectively. Set 
	\begin{equation}\label{eq_S}
		S \defeq \{ x + iy \in \C \colon |y| < 1/2 \},
	\end{equation}
	and let $\phi \colon V \rightarrow S$ be a conformal isomorphism such that each $\phi(\gamma_t)$
	and $\phi(\gamma_{t'})$ connects the upper and lower
	boundaries of the strip $S$, and such that $\phi(\gamma_t)$ separates $\phi(\gamma_{t'})$ from $-\infty$ in $S$. For $t\leq s \leq t'$, let $\theta(s)$ denote the shortest length of a vertical line segment at real part $s$ that separates $\gamma_t$ from $\gamma_{t'}$ in $V$. If $\int_{t}^{t'}ds/\theta(s) \geq 1/2$, then 
	\begin{equation}
		\min_{z \in \gamma_{t'}} \Rea \phi(z)- \max_{w \in \gamma_{t}} \Rea \phi(w)\geq \int_{t}^{t'}\frac{ds}{\theta(s)}- \frac{1}{\pi}\ln 32.
	\end{equation}
\end{theorem}

Before applying Ahlfors’ distortion theorem, we record an elementary estimate on the conformal map $\psi:S\to H$, which will play a key role in translating hyperbolic estimates into Euclidean growth estimates.

\begin{lemma}\label{lem:psi_basic_estimates}
Let \(S\) be as in \eqref{eq_S}, and let \(\psi:S\to H\) be the conformal isomorphism with \(\psi(0)=1\) and \(\psi'(0)>0\). Then there are constants \(a_\psi,A_\psi>0\), depending only on \(H\), such that 
	\begin{equation}\label{eq:lem_aux1}
e^{a_\psi u}\psi(s)\leq\psi(s+u)\leq e^{A_\psi u}\psi(s) \quad \text{ for all } s\in\R \text{ and } u\geq0.
	\end{equation}
Moreover, for every \(s\in\R\), \(\max_{|y|<1/2}\Rea\psi(s+iy)=\psi(s)\).
\end{lemma}

\begin{proof}
Since both \(S\) and \(H\) are symmetric with respect to the real axis, the normalization gives \(\psi(\overline z)=\overline{\psi(z)}\). Thus \(\psi(\R)\subset\R\). As \(\psi\) is a conformal isomorphism from \(S\) onto \(H\), and \(\psi(0)=1\), it follows that \(\psi(\R)=(0,\infty)\). Moreover, the condition \(\psi'(0)>0\) fixes the orientation, so \(\psi\) is increasing on \(\R\). For \(x>0\), we have \(\dist(x,\partial H)=x\). Thus \eqref{eq:hypest} gives \(1/(2x)\le\rho_H(x)\le 2/x\). By conformal invariance of the hyperbolic density, \(\rho_H(\psi(s))\psi'(s)=\rho_S(s)=\rho_S(0)\) for \(s\in\R\), since real translations are automorphisms of \(S\). Writing \(\lambda_S\defeq\rho_S(0)\), we obtain
\[
\frac{\lambda_S}{2}\le \frac{\psi'(s)}{\psi(s)}\le 2\lambda_S
	\qquad \text{for all }s\in\R.
\]

Now let \(s\in\R\) and \(u\geq0\). Integrating from \(s\) to \(s+u\), we obtain
		\[
		\frac{\lambda_S}{2}u
		\leq \int_s^{s+u}\frac{\psi'(t)}{\psi(t)}\,dt= \log\frac{\psi(s+u)}{\psi(s)}
		\leq 2\lambda_S u.
		\]
		Exponentiating and denoting \(a_\psi=\lambda_S/2\) and \(A_\psi=2\lambda_S\), \eqref{eq:lem_aux1} follows.
	
The final assertion follows from the claim on vertical crosscuts in \cite[p.~13]{letidave_variations}. Indeed, applying that claim to the vertical segment $\sigma_s=\{s+iy: |y|<1/2\},$ we see that any point where \(\Rea \psi\) attains its maximum on
		\(\sigma_s\) must be real. Since the only real point of \(\sigma_s\) is
		\(s\), this maximum is attained at \(s\), and so
		$
		\max_{|y|<1/2}\Rea \psi(s+iy)=\Rea \psi(s)=\psi(s).
		$
\end{proof} 

	\begin{corollary}
	\label{cor:psi_large_real}Let \(S\) be as in \eqref{eq_S}, and let \(\psi:S\to H\) be the normalised conformal map from Lemma~\ref{lem:psi_basic_estimates}. For \(b\in\R\), set
	\[
	\Psi_b(\xi)\defeq \frac{\psi(b+\xi)}{\psi(b)},
	\qquad \xi\in S .
	\]
	Then, as \(b\to+\infty\),
	$
	\Psi_b(\xi)\longrightarrow e^{\pi \xi}
	$
	locally uniformly in \(S\), with convergence of derivatives. In particular,
	for every \(\rho\in(0,1/2)\),
	$
	\frac{\Rea \psi(b+it)}{\psi(b)}
	\longrightarrow
	\cos(\pi t)
	$
	in \(C^2([-1/2+\rho,1/2-\rho])\).
\end{corollary}

\begin{proof}
By Lemma~\ref{lem:psi_basic_estimates}, we have \(\psi(b)\to+\infty\). Write \(H_b\defeq \psi(b)^{-1}H\). Then
		\[
		H_b=
		\left\{X+iY:
		X>-\frac{14}{\psi(b)}\log_+(\psi(b)|Y|)
		\right\},
		\]
		and \(H_b\) converges, as \(b\to+\infty\), in the Carath\'eodory kernel sense, with base point \(1\), to the right half-plane \(\HH\). The maps \(\Psi_b:S\to H_b\) satisfy \(\Psi_b(0)=1\) and \(\Psi_b'(0)>0\). Hence the kernel theorem gives convergence to the unique normalised conformal map \(S\to\HH\), namely \(\xi\mapsto e^{\pi\xi}\).
	The convergence of derivatives follows by standard local uniform convergence
	of conformal maps. Taking real parts on compact subintervals of the central
	line gives the stated \(C^2\)-convergence.
\end{proof}

In our constructions, we will use a strategy similar to that in \cite[Section~5]{letidave_variations}, exploiting the geometry of the tracts together with Ahlfors-type distortion estimates to obtain growth bounds and hence prove that the resulting functions have finite order. The following lemma abstracts the key argument used to achieve this.

\begin{lemma}
	\label{lem:finite-order-criterion}
	Let $(G,V,H)$ be a model, with $H$ as in \eqref{eq_H}, \(-1\in V\), and \(G(-1)=1\), and let
	$f\in\mathcal B$ be the approximating entire function, with error term $h$, given by
	Theorem~\ref{thm_approx1.7}. Let $S$ be as in \eqref{eq_S}, let $\psi\colon S\to H$ be the conformal isomorphism with
	$\psi(0)=1$ and $\psi'(0)>0$, and set
	\[
	\phi\defeq \psi^{-1}\circ G\colon V\to S.
	\]
	Suppose that there are positive constants $t_0, A, B, C$, with $B > 2 \log 32$, such that the following two conditions hold for all $t > t_0$. First, there is a point $\zeta \in V$ such that 
	\[
	t\le \Rea \zeta 
	\qquad\text{and}\qquad
	d_V(-1,\zeta)\le At.
	\]
	Second, there exists $t' \in [t + B, t + B + C]$ such that $V \cap \{z\in\C:\Rea z=t'\}$ is connected.
	Then $f$ has finite order.
	
Moreover, for every \(T>0\), there is a constant
\(L_T>0\), depending only on \(T,t_0,A,B,C\) and \(H\), such that
$
\Rea G(z)\le L_T
$
for every \(z\in V\) with \(\Rea z\le T\).
\end{lemma}

\begin{proof}
Fix \(t>t_0\). By hypothesis, choose
\(t_1\in[t+B,t+B+C]\) and \(t_2\in[t_1+B,t_1+B+C]\) such that
\[
\gamma_j\defeq V\cap\{z\in\C:\Rea z=t_j\},\qquad j=1,2,
\]
are connected. Then \(\gamma_1\) and \(\gamma_2\) are vertical crosscuts
of \(V\), and \(\phi(\gamma_1)\) and \(\phi(\gamma_2)\) connect the two
boundary components of \(S\). Moreover, \(\phi(\gamma_1)\) separates
\(\phi(\gamma_2)\) from \(-\infty\) in \(S\).

If \(\theta(s)\) denotes the shortest length of a vertical segment at real
part \(s\) separating \(\gamma_1\) from \(\gamma_2\) in \(V\), then
\(\theta(s)\le 2\pi\) for \(s\in[t_1,t_2]\), since \(V\) is disjoint from
its \(2\pi i\)-translates. Hence
\[
\int_{t_1}^{t_2}\frac{ds}{\theta(s)}
\ge \frac{t_2-t_1}{2\pi}
\ge \frac{B}{2\pi}
> \frac{\log 32}{\pi}> \frac12.
\]
Thus the hypothesis of Theorem~\ref{thm_ahlfors} is satisfied, and the theorem gives
\[
\min_{z\in\gamma_2}\Rea\phi(z)
-
\max_{w\in\gamma_1}\Rea\phi(w)>0.
\]
Thus there exists a vertical crosscut \(\sigma=\{c+iy:|y|<1/2\}\) of \(S\) separating \(\phi(\gamma_1)\) from \(\phi(\gamma_2)\). In particular, \(\phi^{-1}(\sigma)\) separates every point of \(V\) with real part at most \(t\), and in particular \(-1\), from the end at \(+\infty\) in \(V\). 

We now apply the first hypothesis with \(t_2\) in place of \(t\). Then there exists a point \(\eta\in V\) such that \(\Rea\eta\ge t_2\) and \(d_V(-1,\eta)\le At_2\). The hyperbolic geodesic joining \(-1\) to \(\eta\) must meet \(\phi^{-1}(\sigma)\); let \(p\) be its first point of intersection with \(\phi^{-1}(\sigma)\). Then
\[
d_V(-1,p)\le d_V(-1,\eta)\le At_2\le A(t+2(B+C)).
\]
Since \(\phi^{-1}(\sigma)\) separates \(-1\) from the end at \(+\infty\), we have \(c>0\). If \(\phi(p)=c+iy\), with \(|y|<1/2\), then writing \(\lambda_S\defeq\rho_S(0)\),
\[
\lambda_S c
= d_S(0,c)
\le d_S(0,c+iy)
= d_V(-1,p)
\le A(t+2(B+C)).
\]

Set \(\mu\defeq\max_{z\in\sigma}\Rea\psi(z)\). By Lemma~\ref{lem:psi_basic_estimates}, \(\mu=\psi(c)\). 
Using Lemma~\ref{lem:psi_basic_estimates} once more, we obtain
\[
\mu=\psi(c)\le \exp(A_\psi c)
\le \exp\left(\frac{A_\psi A}{\lambda_S}(t+2(B+C))\right),
\]
where \(A_\psi\) is the constant from Lemma~\ref{lem:psi_basic_estimates}.

If \(z\in V\) and \(\Rea z\le t\), write \(\phi(z)=x+iy\). Since \(\phi(z)\) lies in the component \(\{\Rea w<c\}\) of \(S\setminus\sigma\), we have \(x<c\). By Lemma~\ref{lem:psi_basic_estimates},
$
\Rea\psi(x+iy)\le \psi(x).
$
Since \(\psi\) is increasing on \(\R\), it follows that \(\psi(x)\le\psi(c)\).
Thus
\begin{equation}\label{eq:lemma_finite_order}
	\Rea G(z)=\Rea\psi(\phi(z))\le\psi(c)=\mu.
\end{equation}
Let \(r=e^t\). Hence, for every \(z\in V\) with \(\Rea z=t\), since \(g(e^z)=\exp(G(z))\), we have
\[
|g(e^z)|=\exp(\Rea G(z))\le \exp(\mu).
\]
Now let \(\xi\in\C\) with \(|\xi|=r\). If \(\xi\in\exp(V)\), then, since \(V\) is disjoint from its \(2\pi i\)-translates, there is a unique \(z\in V\) such that \(e^z=\xi\), and necessarily \(\Rea z=t\). Therefore Theorem~\ref{thm_approx1.7}, with \(M\) denoting the constant appearing there, gives
\[
|f(\xi)|
\le |g(\xi)|+|h(\xi)|
\le \exp(\mu)+M|\xi|^{-1}
= \exp(\mu)+Me^{-t}.
\]
On the other hand, if \(\xi\notin\exp(V)\), then the \(g\)-term is absent, and Theorem~\ref{thm_approx1.7} gives
\[
|f(\xi)|=|h(\xi)|\le M|\xi|^{-1}=Me^{-t}
\le \exp(\mu)+Me^{-t}.
\]
Thus, in both cases,
$
|f(\xi)|\le \exp(\mu)+Me^{-t}.
$
Consequently,
\[
M(r,f)
\le
\exp(\mu)+Me^{-t}
\le
\exp\left(\exp\left(\frac{A_\psi A}{\lambda_S}(t+2(B+C))\right)\right)
+Me^{-t}.
\]
It follows that
\[
\log\log M(r,f)
\le
\frac{A_\psi A}{\lambda_S}t+O(1)
=
\frac{A_\psi A}{\lambda_S}\log r+O(1)
\]
as \(r\to\infty\). Hence \(f\) has finite order.

Finally, let \(T>0\), and put \(t_T\defeq\max\{T,t_0+1\}\). Applying
\eqref{eq:lemma_finite_order} with \(t=t_T\), we obtain
\[
\Rea G(z)\le \psi\left(\frac{A(t_T+2(B+C))}{\lambda_S}\right)
\]
for every \(z\in V\) with \(\Rea z\le T\). Thus the final assertion follows
with
\[
L_T\defeq
\psi\left(\frac{A(\max\{T,t_0+1\}+2(B+C))}{\lambda_S}\right).
\]
This constant depends only on \(T,t_0,A,B,C\) and \(H\), as required.
\end{proof}
\section{Proof of Theorem~\ref{th:2}}\label{S.th2}

Let \((z_n)\) be a sequence as in Theorem~\ref{th:2}. If one of the points is \(0\), we omit it from the construction, since \(0\in\M(f)\) for every entire function \(f\). By relabelling the remaining points, we may assume that \(0<|z_n|<|z_{n+1}|\) for all \(n\in\N\). For each \(n\in\N\), choose an argument \(\arg(z_n)\in(-\pi,\pi]\) and define
\[
w_n \defeq x_n+iy_n \defeq \log|z_n| + i\,\arg(z_n).
\]
(Any other choice of $\arg(z_n)$ changes $w_n$ by an integer multiple of $2\pi i$. This will not affect the construction, since $\Mlog(f)$ is $2\pi i$--periodic; see \eqref{eq_max_translates}.)
In particular, $e^{w_n}=z_n$ for all $n$.

We will prove our result by constructing $f\in\mathcal B$ such that $w_n\in \Mlog(f)$ for all $n\in\N$; then \eqref{eq_max_translates} implies that $z_n\in\M(f)$ for all $n$. We shall next define our collections of domains $V(\uldelta)$, for some parameters $\uldelta$. 
For $\ell > 0$, consider the rectangle
\begin{align*}
	R_0 = R_0(\ell) \defeq \{ x + iy \in \C \colon -2 < x < 0 \text{ and } |y| < \ell \}.
\end{align*}

\begin{observation}
	\label{claim:ell}
	We can choose $\ell$ sufficiently small that the following holds. Suppose that $V \subset \C$ is an unbounded simply-connected domain, which strictly contains $R_0$, and is such that the upper, lower and left-hand boundaries of $R_0$ all lie in $\partial V$. Suppose also that $G \colon V \to H$ is a conformal map with $G(-1) = 1$, and also such that $G(z) \rightarrow \infty$ as $\Rea z \rightarrow +\infty$. Then 
	\[
	\max_{z \in \partial R_0 \cap V} \Rea G(z) \geq 4. 
	\]
\end{observation}
\begin{proof}
	Let \(J\defeq\partial R_0\cap V\). The set \(J\) is non-empty: otherwise \(R_0\) would be both open and closed in \(V\), contradicting that \(V\) is connected and strictly contains \(R_0\). Set \(C_0\defeq d_H(1,4)<\infty\). Since \(H\) is symmetric with respect to the real axis, the interval \([1,4]\) lies on the symmetry geodesic, and hence \(d_H(1,t)\le C_0\) for \(1\le t\le4\). Since \(J\) separates \(-1\) from the end at infinity in \(V\), \(G(J)\) separates \(1\) from infinity in \(H\). If \(\max_J\Rea G<4\), then \(G(J)\) must meet \([1,4]\); otherwise \([1,\infty)\subset H\) would connect \(1\) to infinity without crossing \(G(J)\). Hence, by conformal invariance, \(d_V(-1,J)=d_H(1,G(J))\le C_0\). Conversely, every curve from \(-1\) to \(J\) contains, before first hitting \(J\), a subcurve in \(R_0\) of Euclidean length at least \(1\). Since \(\dist(z,\partial V)\le\ell\) on \(R_0\), \eqref{eq:hypest} gives \(d_V(-1,J)\ge 1/(2\ell)\). Choosing \(\ell<1/(2C_0)\) gives a contradiction.
\end{proof}

From now on we assume that \(\ell\) has been chosen so that the implication of Observation~\ref{claim:ell} holds for any domain \(V\) as in the observation. Let \(M>0\) be the constant from Theorem~\ref{thm_approx1.7} for such models, normalised by \(G^{-1}(1)=-1\). By Theorem~\ref{thm_approx1.7}, under this normalisation the same constant \(M\) may be used for all the domains and conformal maps considered here. Increasing \(M\) if necessary, we can assume that \(M>e^2\). Set
\begin{equation}\label{eq_L}
	L\defeq \log M, \quad \text{ so that} \quad L>2.
\end{equation}

After an initial rescaling, we may assume without loss of generality that \(x_1=\Rea w_1\) is larger than any fixed constant required below; in particular, we assume that
\[
\Rea w_1>L+3.
\]
Indeed, replacing the prescribed sequence \((z_n)\) by \((cz_n)\), with \(c>0\), replaces each logarithmic lift \(w_n\) by \(w_n+\log c\). Once the construction has been carried out for the rescaled sequence, let \(\tilde{f}\) be the resulting function and set \(f(z)\defeq \tilde{f}(cz)\). This gives the required function for the original sequence. Indeed, \(M(r,f)=M(cr,\tilde{f})\), so \(cz_n\in\M(\tilde{f})\) implies \(z_n\in\M(f)\). This rescaling preserves membership of \(\B\), and it also preserves finite order.

Let \(\vartheta:[x_1,\infty)\to\R\) be the continuous piecewise affine
function such that \(\vartheta(x_n)=y_n\) for all \(n\in\N\), and put
$
\gamma=\{x+i\vartheta(x):x\ge x_1\}.
$
Let \(T\) be the open trapezoidal region bounded by the two segments
\([-i\ell,w_1-i]\) and \([i\ell,w_1+i]\), and by the vertical segments
\(\{iy:|y|\le\ell\}\) and \(\{x_1+iy:|y-y_1|\le1\}\).
Define the domain
\[
U\defeq \operatorname{int}\left(
\overline{R_0}\cup\overline T\cup
\{x+iy:x\ge x_1,\ |y-\vartheta(x)|\le1\}
\right).
\]
Then \(U\) is open, simply connected, and, after decreasing \(\ell\) if
necessary, disjoint from its non-zero \(2\pi i\)-translates. Moreover,
$
U\cap\{\Rea z=x_n\}=\{x_n+iy:|y-y_n|<1\}.
$
All tracts in this part of the construction will be obtained from this fixed domain \(U\) by cutting along vertical slits and leaving one short gate at each line \(\Rea z=x_n\).

We call a sequence \(\uleta=(\eta_n)_{n\in\N}\) \emph{admissible} if \(0<\eta_n<1/10\), the discs \(\overline{D(w_n,20\eta_n)}\) are contained in \(U\), are pairwise disjoint, and meet no vertical line \(\{\Rea z=x_m\}\) with \(m\ne n\). For an admissible \(\uleta\), set
\[
\Delta_{\uleta}\defeq\prod_{n\in\N}[0,3\eta_n].
\]
For \(\uldelta=(\delta_n)_{n\in\N}\in\Delta_{\uleta}\), define
\[
I_{n,\uleta}(\uldelta)
\defeq
\{x_n+iy:y_n-2\eta_n+\delta_n<y<y_n-\eta_n+\delta_n\},
\]
and
\[
L_{n,\uleta}(\delta_n)
\defeq
\bigl(U\cap\{\Rea z=x_n\}\bigr)\setminus I_{n,\uleta}(\uldelta).
\]
Finally set
\[
V_{\uleta}(\uldelta)\defeq U\setminus\bigcup_{n\in\N}L_{n,\uleta}(\delta_n).
\]
By definition, \(V_{\uleta}(\uldelta)\) is open, simply connected and disjoint from its non-zero \(2\pi i\)-translates. 

\begin{figure}[htp]
	\centering
	\def\svgwidth{0.9\linewidth}
	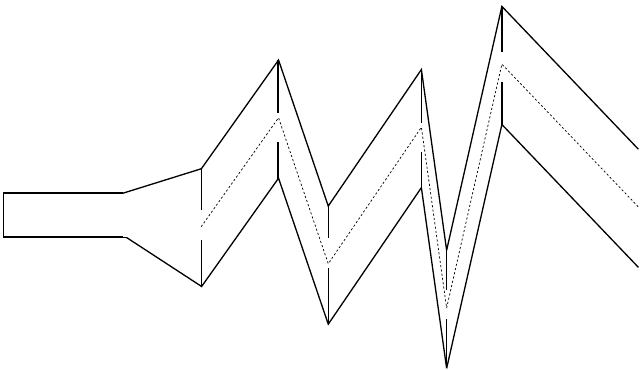
	\caption{Schematic of the fixed domain \(U\) and of one tract
		\(V_{\uleta}(\uldelta)\). At the line \(\Rea z=x_n\), the tract
		\(V_{\uleta}(\uldelta)\) intersects the line in the short gate
		\(I_{n,\uleta}(\uldelta)\). Varying
		\(\delta_n\) moves this gate vertically near \(w_n\), while \(U\) remains
		fixed.} \label{fig:tract_points}
\end{figure}

For \(\uldelta\in\Delta_{\uleta}\), let \(G_{\uleta,\uldelta}\) be the conformal isomorphism from \(V_{\uleta}(\uldelta)\) to \(H\), normalised by \(G_{\uleta,\uldelta}(-1)=1\) and \(G_{\uleta,\uldelta}(z)\to\infty\) as \(\Rea z\to+\infty\). Let \(g_{\uleta,\uldelta}\) be the corresponding function satisfying \eqref{eq:gG}, for the model \((G_{\uleta,\uldelta},V_{\uleta}(\uldelta),H)\), and let \(f_{\uleta,\uldelta}\) and \(h_{\uleta,\uldelta}\) be the corresponding functions provided by Theorem~\ref{thm_approx1.7}.

In what follows, \emph{maximiser} of a real-valued function on a compact segment means a point at which the maximum over that segment is attained.

\begin{lemma}[Choice of the gate widths]\label{lem:narrow_gate_unique}
	There exists an admissible sequence \(\uleta=(\eta_n)_{n\in\N}\) such that, for every \(n\in\N\) and every \(\uldelta\in\Delta_{\uleta}\), the function \(z\mapsto |f_{\uleta,\uldelta}(e^z)|\) has a unique maximiser on \(\overline{I_{n,\uleta}(\uldelta)}\). This maximiser belongs to \(I_{n,\uleta}(\uldelta)\).
	
	Moreover, if there is \(K_0>0\) such that \(x_{n+1}-x_n\ge K_0\) for all \(n\), then, after the initial rescaling above, the sequence \(\uleta\) may be chosen to be constant: \(\eta_n=\eta_0\) for all \(n\), for some \(\eta_0>0\).
\end{lemma}
\begin{proof}
	We choose the numbers \(\eta_n\), and hence the widths of the gates, one at a time. Fix \(n\), and suppose that the previous
	gates have already been chosen. We first choose \(\eta_n\) small enough that
	\(\overline{D(w_n,20\eta_n)}\subset U\), that this disc is disjoint from all
	the other such discs already chosen, and that it meets no line
	\(\{\Rea z=x_m\}\), \(m\ne n\). We shall decrease \(\eta_n\) finitely many
	times below. For \(\uldelta\in\Delta_{\uleta}\), let
	$
	c_{n,\uldelta}
	=
	x_n+i(y_n-3\eta_n/2+\delta_n),
	$
	so that
	$
	I_{n,\uleta}(\uldelta)
	=
	\{c_{n,\uldelta}+i\eta_ns:|s|<1/2\}.
	$
	Set
	$
	\Omega_{n,\uldelta}
	=
	\eta_n^{-1}(V_{\uleta}(\uldelta)-c_{n,\uldelta}).
	$
	Let \(S=\{z:|\Ima z|<1/2\}\), let \(\psi:S\to H\) be the normalised map from
	Lemma~\ref{lem:psi_basic_estimates}, and write
	$
	\Phi_{\uleta,\uldelta}
	=
	\psi^{-1}\circ G_{\uleta,\uldelta}
	:
	V_{\uleta}(\uldelta)\to S .
	$
	Define
	\[
	F_{n,\uldelta}(\zeta)
	=
	\Phi_{\uleta,\uldelta}(c_{n,\uldelta}+\eta_n\zeta),
	\qquad
	\zeta\in\Omega_{n,\uldelta}.
	\]
	Thus \(F_{n,\uldelta}\) is the strip coordinate near the \(n\)-th gate,
	written in the rescaled variable. In this variable the gate is the fixed
	segment \(\{is:|s|<1/2\}\), and, as \(\eta_n\to0\), the rescaled domains
	\(\Omega_{n,\uldelta}\) converge in the Carath\'eodory kernel sense, with base point \(0\), to the two-sided slit plane
	\[
	\Omega
	=
	\C\setminus i\bigl((-\infty,-1/2]\cup[1/2,\infty)\bigr).
	\]
	Indeed, for the fixed \(n\), put
		\(d_n\defeq\min\{\dist(w_n,\partial U),\inf_{m\ne n}|x_m-x_n|\}>0\).
		Fix \(R>0\). Since \(|c_{n,\uldelta}-w_n|\le 3\eta_n/2\), the condition
		\((R+3/2)\eta_n<d_n\) implies, uniformly in
		\(\uldelta\in\Delta_{\uleta}\), that \(D(c_{n,\uldelta},R\eta_n)\subset U\)
		and that this disc meets no line \(\{\Rea z=x_m\}\), \(m\ne n\). Hence, after
		rescaling by \(z\mapsto (z-c_{n,\uldelta})/\eta_n\), the only removed part
		visible in \(D(0,R)\) is the \(n\)-th slit, and therefore
		\(\Omega_{n,\uldelta}\cap D(0,R)=\Omega\cap D(0,R)\) for all sufficiently
		small \(\eta_n\), uniformly in \(\uldelta\). Since \(R\) was arbitrary, the
		claimed Carath\'eodory kernel convergence follows.
	
	By Carath\'eodory kernel convergence for the two-sided slit, there are real
	numbers \(b_{n,\uldelta}\), tending to \(+\infty\) as \(\eta_n\to0\),
	uniformly for \(\uldelta\in\Delta_{\uleta}\), such that
	$
	F_{n,\uldelta}-b_{n,\uldelta}
	\longrightarrow
	\Theta
	$
	locally uniformly in \(\Omega\), with convergence of derivatives, where
	\[
	\Theta(\zeta)=
	\frac1\pi\operatorname{arsinh}(2\zeta)
	=
	\frac1\pi
	\log\left(2\left(\zeta+\sqrt{\zeta^2+1/4}\right)\right).
	\]
	Here \(\operatorname{arsinh}\) denotes the branch that is real on the real axis, the square-root branch is positive on the real axis, and the logarithm is real on the positive real axis. This map satisfies
	\[
	\Theta(0)=0,\qquad \Theta(\R)=\R,
	\qquad
	\Theta(x)\to\pm\infty
	\quad\text{as }x\to\pm\infty .
	\]
	
	On the gate of \(\Omega\), if \(\zeta=is\) and \(|s|<1/2\), then
	$
	\Theta(is)
	=
	\frac{i}{\pi}\arcsin(2s).
	$
	Consequently, as \(\eta_n\to0\), on compact subsegments of the gate,
	$
	F_{n,\uldelta}(is)
	=
	b_{n,\uldelta}
	+
	\frac{i}{\pi}\arcsin(2s)
	+ o(1),
	$
	uniformly in \(\uldelta\), with convergence of derivatives.
	
	Denote
	$
	\tau(s)\defeq \frac1\pi\arcsin(2s).
	$
	Since \(G_{\uleta,\uldelta}=\psi\circ\Phi_{\uleta,\uldelta}\), we have
	$
	\Rea G_{\uleta,\uldelta}(c_{n,\uldelta}+i\eta_ns)
	=
	\Rea \psi(F_{n,\uldelta}(is)).
	$
	By Corollary~\ref{cor:psi_large_real}, applied with
	\(b=b_{n,\uldelta}\), and by the preceding convergence of
	\(F_{n,\uldelta}-b_{n,\uldelta}\), we obtain
	\[
	P_{n,\uldelta}(s)
	\defeq
	\frac{
		\Rea G_{\uleta,\uldelta}(c_{n,\uldelta}+i\eta_ns)
	}{
		\psi(b_{n,\uldelta})
	}
	\longrightarrow
	\cos(\pi\tau(s))=
	\cos(\arcsin(2s))
	=
	2\sqrt{1/4-s^2}\eqdef	p(s)
	\]
	in \(C^2\) on compact subintervals of \((-1/2,1/2)\), uniformly in
	\(\uldelta\).

We also need a \(C^0\) version of this convergence at the two endpoints of
		the gate. 	We explain the upper endpoint; the lower one is analogous. In a
		neighbourhood of \(i/2\), use the coordinate
		\(\xi=\sqrt{-i(\zeta-i/2)}\), with the branch chosen so that the slit
		\(i[1/2,\infty)\) is opened into the boundary of a half-disc. In this
		coordinate the endpoint becomes a regular analytic boundary point, and, for
		small \(\eta_n\), the rescaled domains agree with the model domain near this
		point, uniformly in \(\uldelta\). Hence the boundary form of the
		Carath\'eodory convergence theorem gives uniform convergence of
		\(F_{n,\uldelta}-b_{n,\uldelta}\) to \(\Theta\) up to the endpoint. Applying
		the same argument at \(-i/2\), we obtain \(P_{n,\uldelta}\to p\) uniformly on
		the closed gate, where \(p(\pm1/2)=0\). Choose \(\rho>0\) so small that
	\(p(s)<1/4\) whenever \(1/2-\rho<|s|<1/2\). After decreasing \(\eta_n\),
	we therefore have
	\[
	P_{n,\uldelta}(s)<\frac12
	\qquad
	\text{whenever }1/2-\rho<|s|<1/2,
	\]
	uniformly in \(\uldelta\). It remains to look at the compact interval
	\[
	K_\rho\defeq[-1/2+\rho,1/2-\rho].
	\]
	There the convergence \(P_{n,\uldelta}\to p\) is in \(C^2\). Since \(p\) has
	a unique non-degenerate maximum at \(0\), the functions
	\(P_{n,\uldelta}\) also have exactly one maximum on \(K_\rho\), after
	decreasing \(\eta_n\). Moreover, since \(p(0)=1\), this maximum has value
	larger than \(3/4\), after decreasing \(\eta_n\). The endpoint estimate
	therefore rules out maxima in the two endpoint regions. Hence
	\(P_{n,\uldelta}\), and therefore
	$
	s\mapsto
	\Rea G_{\uleta,\uldelta}(c_{n,\uldelta}+i\eta_ns),
	$
	has a unique maximum on \((-1/2,1/2)\). Since \(r\mapsto e^r\) is strictly
	increasing, the same is true for
	$
	s\mapsto
	|g_{\uleta,\uldelta}(e^{c_{n,\uldelta}+i\eta_ns})|.
	$
	
	To conclude the proof, we must pass from the model function to the entire
	function. By Theorem~\ref{thm_approx1.7}, applied in logarithmic
	coordinates,
	$
	|h_{\uleta,\uldelta}(e^z)|\le Me^{-\Rea z}
	$
	whenever \(\Rea z>L\). Since the gate lies on the line \(\Rea z=x_n\), and
	since \(x_n-1>L\), this gives in particular
	that $
	|h_{\uleta,\uldelta}(e^{c_{n,\uldelta}+i\eta_ns})|
	\le Me^{-x_n},$ whenever $-1/2\le s\le1/2$. Let
	\[
	u(s)=\Rea G_{\uleta,\uldelta}(c_{n,\uldelta}+i\eta_ns).
	\]
	The estimates used below are consequences of the preceding
		\(C^2\)-convergence \(P_{n,\uldelta}\to p\), together with the definition of
		\(P_{n,\uldelta}\). Choose \(0<\alpha<1/2-\rho\) and
		\(d,\kappa,c>0\), independent of \(\uldelta\), such that, for the limit
		function \(p\), we have \(p''\le -2\kappa\) and \(p\ge 2c\) on
		\([-\alpha,\alpha]\), while \(p\le1-3d\) on
		\(K_\rho\setminus(-\alpha,\alpha)\). Combining this with the endpoint
		estimate above, and decreasing \(\eta_n\) if necessary, we obtain uniformly
		in \(\uldelta\) that the unique maximum point \(s_*=s_*(\uldelta)\) lies in
		\((-\alpha,\alpha)\), that
		\(P_{n,\uldelta}(s)\le P_{n,\uldelta}(s_*)-d\) for
		\(\alpha\le |s|<1/2\), that \(P_{n,\uldelta}''(s)\le-\kappa\) for
		\(|s|\le\alpha\), and that \(P_{n,\uldelta}(s)\ge c\) for
		\(|s|\le\alpha\). Since \(u\) is \(\psi(b_{n,\uldelta})\) times
		\(P_{n,\uldelta}\) by the definition of \(P_{n,\uldelta}\), these give
		\[
			\begin{aligned}
				u(s)&\le u(s_*)-d\psi(b_{n,\uldelta})
				&&\text{for } \alpha\le |s|<1/2,\\
				u''(s)&\le -\kappa\psi(b_{n,\uldelta})
				&&\text{for } |s|\le\alpha,\\
				u(s)&\ge c\psi(b_{n,\uldelta})
				&&\text{for } |s|\le\alpha .
		\end{aligned}
		\]
	Let
	\(r_\alpha\defeq(1/2-\alpha)/2\). After decreasing $\eta_n$, the points $c_{n,\uldelta}+i\eta_ns$, with $\operatorname{dist}(s,[-\alpha,\alpha])\le r_\alpha$, belong to $V_{\uleta}(\uldelta)$ and have real part at least $x_{n-1}$, uniformly in $\uldelta$. Hence Cauchy's estimates applied to the holomorphic function
	\(s\mapsto h_{\uleta,\uldelta}(e^{c_{n,\uldelta}+i\eta_ns})\) give, for
	\(k=0,1,2\),
	\[
	\left|
	\frac{d^k}{ds^k}
	h_{\uleta,\uldelta}
	\bigl(e^{c_{n,\uldelta}+i\eta_ns}\bigr)
	\right|
	\le
	C_\alpha Me^{-x_n},
	\qquad |s|\le\alpha,
	\]
	with \(C_\alpha\) independent of \(\uldelta\). 
	On \([-\alpha,\alpha]\), we have
	\[
	\begin{aligned}
		f_{\uleta,\uldelta}
		\bigl(e^{c_{n,\uldelta}+i\eta_ns}\bigr)
		&=
		e^{G_{\uleta,\uldelta}(c_{n,\uldelta}+i\eta_ns)}
		\left(
		1+
		h_{\uleta,\uldelta}
		\bigl(e^{c_{n,\uldelta}+i\eta_ns}\bigr)
		e^{-G_{\uleta,\uldelta}(c_{n,\uldelta}+i\eta_ns)}
		\right).
	\end{aligned}
	\]
	The term in parentheses is a \(C^2\)-small perturbation of \(1\), uniformly
	in \(\uldelta\), as \(\eta_n\to0\). Indeed, the \(h\)-term and its first two
	\(s\)-derivatives are \(O(Me^{-x_n})\), while on \([-\alpha,\alpha]\)
	$
	\Rea G_{\uleta,\uldelta}(c_{n,\uldelta}+i\eta_ns)
	\ge
	c\psi(b_{n,\uldelta})
	$
	and
	\[
	\frac{d^j}{ds^j}
	G_{\uleta,\uldelta}(c_{n,\uldelta}+i\eta_ns)
	=
	O(\psi(b_{n,\uldelta}))
	\qquad (j=1,2).
	\]
	Thus differentiating the factor
	\(e^{-G_{\uleta,\uldelta}(c_{n,\uldelta}+i\eta_ns)}\) produces only polynomial
	factors in \(\psi(b_{n,\uldelta})\), whereas
	$
	e^{-\Rea G_{\uleta,\uldelta}(c_{n,\uldelta}+i\eta_ns)}
	\le
	e^{-c\psi(b_{n,\uldelta})}.
	$
	Since \(\psi(b_{n,\uldelta})\to+\infty\), the perturbation tends to \(0\) in
	\(C^2([-\alpha,\alpha])\), uniformly in \(\uldelta\). Therefore
	$
	s\mapsto
	\log\left|
	f_{\uleta,\uldelta}
	\bigl(e^{c_{n,\uldelta}+i\eta_ns}\bigr)
	\right|
	$
	is a \(C^2\)-small perturbation of \(u(s)\) on \([-\alpha,\alpha]\). Since
	\(u''\le-\kappa\psi(b_{n,\uldelta})\) there, this function is strictly
	concave after decreasing \(\eta_n\). Hence
	$
	s\mapsto
	\left|
	f_{\uleta,\uldelta}
	\bigl(e^{c_{n,\uldelta}+i\eta_ns}\bigr)
	\right|
	$
	has at most one maximum in \((-\alpha,\alpha)\).
	
	It remains to rule out points outside \((-\alpha,\alpha)\). For
	\(\alpha\le |s|<1/2\),
	\[
	\left|
	e^{G_{\uleta,\uldelta}(c_{n,\uldelta}+i\eta_ns)}
	\right|
	\le
	e^{-d\psi(b_{n,\uldelta})}
	\left|
	e^{G_{\uleta,\uldelta}(c_{n,\uldelta}+i\eta_ns_*)}
	\right|.
	\]
	Moreover
	\[
	\left|
	e^{G_{\uleta,\uldelta}(c_{n,\uldelta}+i\eta_ns_*)}
	\right|
	=
	e^{u(s_*)}
	\ge
	e^{c\psi(b_{n,\uldelta})}.
	\]
	Thus, since the additive error coming from \(h_{\uleta,\uldelta}\) is bounded
	by \(CMe^{-x_n}\), after decreasing \(\eta_n\) no point with
	\(\alpha\le |s|<1/2\) can give a value as large as the value at \(s_*\).
	Consequently the maximum on the open gate is attained in
	\((-\alpha,\alpha)\), where it is unique.
	
	Finally, let \(z\) be one of the two endpoints of
	\(\overline{I_{n,\uleta}(\uldelta)}\). Then
	\(z\in L_{n,\uleta}(\delta_n)\), and hence
	\(e^z\notin\exp(V_{\uleta}(\uldelta))\). By the definition
	\eqref{eq_f} of the approximating function, the model term is absent at
	\(e^z\), and therefore
	\[
	|f_{\uleta,\uldelta}(e^z)|
	=
	|h_{\uleta,\uldelta}(e^z)|
	\le
	Me^{-x_n}.
	\]
	At the interior point corresponding to \(s_*\), however,
	\[
	\left|
	f_{\uleta,\uldelta}
	\bigl(e^{c_{n,\uldelta}+i\eta_ns_*}\bigr)
	\right|
	\ge
	e^{c\psi(b_{n,\uldelta})}
	-
	CMe^{-x_n}.
	\]
	After decreasing \(\eta_n\), this is larger than \(Me^{-x_n}\). Therefore
	neither endpoint can be a maximiser. Since the open gate has a unique
	maximiser, the closed gate
	\(\overline{I_{n,\uleta}(\uldelta)}\) also has a unique maximiser, and this
	maximiser lies in \(I_{n,\uleta}(\uldelta)\).
	
	Proceeding inductively over \(n\), and decreasing each \(\eta_n\) as above,
	gives the required admissible sequence.
	
	It remains to justify the final assertion. Suppose that \(x_{n+1}-x_n\ge K_0>0\) for all \(n\). Since \(y_n\in(-\pi,\pi]\), the slopes of the line segments forming \(\gamma\) are bounded in terms of \(K_0\). After the initial rescaling, the same is true, with a uniform bound, for the two sides of the initial trapezoid \(T\). Hence there is \(r_0>0\) such that \(D(w_n,r_0)\subset U\) for every \(n\). Choose \(\eta_0>0\) so small that $0<\eta_0<1/10$, \(20\eta_0<r_0\) and \(40\eta_0<K_0\). Then the constant sequence \(\eta_n=\eta_0\) is admissible.
	
	We claim that, after decreasing \(\eta_0\) if necessary, the uniqueness
	argument above holds uniformly for all \(n\) and all
	\(\uldelta\in\prod_n[0,3\eta_0]\). Fix \(R>0\). Since
	\(|c_{n,\uldelta}-w_n|\le3\eta_0/2\), if \(\eta_0\) is sufficiently small,
	depending only on \(R,r_0,K_0\), then
	\(D(c_{n,\uldelta},R\eta_0)\subset U\) and this disc meets no line
	\(\{\Rea z=x_m\}\) with \(m\ne n\), for every \(n\) and every
	\(\uldelta\). Hence, in the rescaled coordinate
	\(\zeta=(z-c_{n,\uldelta})/\eta_0\), the only boundary pieces visible in
	\(D(0,R)\) are the two parts of the \(n\)-th slit. Thus
	\[
	\eta_0^{-1}\bigl(V_{\uleta}(\uldelta)-c_{n,\uldelta}\bigr)
	\longrightarrow
	\Omega=\C\setminus i\bigl((-\infty,-1/2]\cup[1/2,\infty)\bigr)
	\]
	in the Carath\'eodory kernel sense, uniformly in \(n\) and
	\(\uldelta\). The corresponding convergence of
	\(F_{n,\uldelta}-b_{n,\uldelta}\) to \(\Theta\), including convergence of
	derivatives on compact subsets of \(\Omega\), is therefore also uniform in
	\(n\) and \(\uldelta\). Moreover \(b_{n,\uldelta}\to+\infty\) uniformly as
	\(\eta_0\to0\), by the same argument as in the preceding single-gate case. Since \(x_n\ge x_1>L+3\), the approximation error
	\(Me^{-x_n}\) is uniformly bounded by \(Me^{-x_1}\). Hence all estimates in
	the uniqueness proof hold simultaneously for every gate once \(\eta_0\) is
	sufficiently small. This proves the constant-width assertion.
\end{proof}

From now on, fix an admissible sequence \(\uleta\) given by Lemma~\ref{lem:narrow_gate_unique}, and suppress \(\uleta\) from the notation. Thus we write \(\Delta\), \(L_n(\delta_n)\), \(V(\uldelta)\), \(I_n(\uldelta)\), \(G_{\uldelta}\), \(g_{\uldelta}\), \(f_{\uldelta}\), and \(h_{\uldelta}\).

It follows from Observation~\ref{claim:ell} that if \(\beta\subset V(\uldelta)\setminus R_0\) is a vertical segment in \(V(\uldelta)\), with endpoints in \(\partial V(\uldelta)\), which separates \(-1\) from infinity, then
\begin{equation}\label{eq:G}
	\max_{z\in\beta}\Rea G_{\uldelta}(z)\ge4.
\end{equation}

\begin{proposition}\label{prop_MaxInV}
	For every \(\uldelta\in\Delta\),
	\[
	\Mlog(f_{\uldelta})\cap\{x+iy:x>L+1\}
	\subset \bigcup_{k\in\Z}\bigl(V(\uldelta)+2k\pi i\bigr).
	\]
\end{proposition}

\begin{proof}
	By the choice of \(L\) in \eqref{eq_L},
	\begin{equation}\label{eq:h}
		|h_{\uldelta}(e^z)|<e^{-1},\qquad \Rea z>L+1.
	\end{equation}
	If \(z\notin\bigcup_{k\in\Z}(V(\uldelta)+2k\pi i)\) and \(\Rea z>L+1\), then \(e^z\notin\exp(V(\uldelta))\), and hence \(|f_{\uldelta}(e^z)|=|h_{\uldelta}(e^z)|<e^{-1}\). On the other hand, fix \(t>L+1\). The cross-section \(V(\uldelta)\cap\{\Rea z=t\}\) is a vertical crosscut separating \(-1\) from infinity. Hence, by \eqref{eq:gG}, \eqref{eq:G} and \eqref{eq:h},
	\[
	\max_{\substack{z\in V(\uldelta)\\ \Rea z=t}} |f_{\uldelta}(e^z)|
	\ge
	\max_{\substack{z\in V(\uldelta)\\ \Rea z=t}} e^{\Rea G_{\uldelta}(z)}-e^{-1}
	\ge e^4-e^{-1}>e^{-1}
	\]
	and the result follows.
\end{proof}

For \(n\in\N\) and \(\uldelta\in\Delta\), let \(a_n(\uldelta)\) be the unique maximiser of \(z\mapsto |f_{\uldelta}(e^z)|\) on \(\overline{I_n(\uldelta)}\), whose existence and uniqueness are given by Lemma~\ref{lem:narrow_gate_unique}. Define
\begin{equation}\label{eq:phi_n}
	\phi_n(\uldelta)\defeq \Ima a_n(\uldelta)-y_n.
\end{equation}

\begin{lemma}\label{lem_props}
	For each \(n\in\N\), the following hold.
	\begin{enumerate}
		\item The function \(\phi_n:\Delta\to\R\) is continuous.
		\item If \(\delta_n=0\), then \(\phi_n(\uldelta)<0\). \label{lem_props:2}
		\item If \(\delta_n=3\eta_n\), then \(\phi_n(\uldelta)>0\). \label{lem_props:3}
		\item If \(\phi_n(\uldelta)=0\), then \(w_n\in\Mlog(f_{\uldelta})\).
	\end{enumerate}
\end{lemma}

\begin{proof}
	Fix \(n\in\N\). We first prove continuity. Let \(\uldelta^{(j)}\to\uldelta\).
	We show that \(a_n(\uldelta^{(j)})\to a_n(\uldelta)\); since
	\(\Delta\) is metrizable, this proves continuity. Put
	\(a_j=a_n(\uldelta^{(j)})\). The points \(a_j\) all lie in the fixed
	compact segment \(\{x_n+iy:y_n-2\eta_n\le y\le y_n+2\eta_n\}\).
	Hence, after passing to a subsequence, \(a_j\to a\). Since
	\(\delta_n^{(j)}\to\delta_n\), the limit satisfies
	\(a\in\overline{I_n(\uldelta)}\). Now take any \(z\in\overline{I_n(\uldelta)}\). The shifted point
	\(z_j\defeq z+i(\delta_n^{(j)}-\delta_n)\) belongs to
	\(\overline{I_n(\uldelta^{(j)})}\) and tends to \(z\). Maximality of
	\(a_j\) gives
	\[
	|f_{\uldelta^{(j)}}(e^{a_j})|\ge |f_{\uldelta^{(j)}}(e^{z_j})|.
	\]
	By Observation~\ref{obs_continuous_dependence},
	\(f_{\uldelta^{(j)}}\to f_{\uldelta}\) locally uniformly. Letting
	\(j\to\infty\) gives \(|f_{\uldelta}(e^a)|\ge |f_{\uldelta}(e^z)|\).
	Since \(z\) was arbitrary, \(a\) is a maximiser on
	\(\overline{I_n(\uldelta)}\), and the uniqueness from
	Lemma~\ref{lem:narrow_gate_unique} gives \(a=a_n(\uldelta)\). Thus every
	subsequence of \(a_n(\uldelta^{(j)})\) has a further subsequence converging to
	\(a_n(\uldelta)\), and therefore \(a_n(\uldelta^{(j)})\to a_n(\uldelta)\).
	This proves the continuity of \(\phi_n\).
	
	If \(\delta_n=0\), then \(I_n(\uldelta)\subset\{z:\Ima z<y_n\}\), so \(\Ima a_n(\uldelta)<y_n\), and hence \(\phi_n(\uldelta)<0\). If \(\delta_n=3\eta_n\), then \(I_n(\uldelta)\subset\{z:\Ima z>y_n\}\), so \(\phi_n(\uldelta)>0\). This proves \eqref{lem_props:2} and \eqref{lem_props:3}.
	
	Finally, assume that \(\phi_n(\uldelta)=0\). Then \(\Ima a_n(\uldelta)=y_n\), while \(a_n(\uldelta)\in I_n(\uldelta)\) has real part \(x_n\). Hence \(a_n(\uldelta)=w_n\). Since \(a_n(\uldelta)\) maximises \(|f_{\uldelta}(e^z)|\) on the base tract slice \(I_n(\uldelta)\), Proposition~\ref{prop_MaxInV} and \(2\pi i\)-periodicity imply that \(w_n\in\Mlog(f_{\uldelta})\).
\end{proof}

\begin{corollary}\label{cor:choose_delta}
	There exists \(\uldelta\in\Delta\) such that \(w_n\in\Mlog(f_{\uldelta})\) for all \(n\in\N\).
\end{corollary}

\begin{proof}
	Apply Lemma \(\ref{lem:shooting}\) to the family $\widetilde\phi_n=-\phi_n$, with $A=\mathbb N, a_n=0, b_n=3\eta_n$, and $\sigma_n=1$ for all $n$. By Lemma \(\ref{lem_props}\)(2),(3), the required sign conditions hold. Hence there is \(\uldelta\in\Delta\) such that \(\phi_n(\uldelta)=0\) for all $n$. Lemma \(\ref{lem_props}\)(4) then gives \(w_n\in\Mlog(f_{\uldelta})\) for all~$n$.
\end{proof}

\begin{proposition}\label{prop:finite_order}
	Suppose that there exists $K_0>0$ such that
	\begin{equation}\label{eq_separation}
		\Rea w_{n+1}\ge \Rea w_n + K_0 \qquad \text{for all }n\in\N.
	\end{equation}
	Then the function $f=f_{\uldelta}$ provided by Corollary~\ref{cor:choose_delta}
	can be chosen to have finite order.
\end{proposition}

\begin{proof}
	We use the final assertion of Lemma~\ref{lem:narrow_gate_unique}. After the initial rescaling, choose the admissible sequence to be constant, \(\eta_n=\eta_0\) for all \(n\). Then apply Corollary~\ref{cor:choose_delta} with this choice of \(\eta\), and let \(\uldelta\in\Delta\) be the resulting parameter. Put \(V\defeq V(\uldelta)\), let \(G\colon V\to H\) be the corresponding conformal map, and let \(f\defeq f_{\uldelta}\) be the approximating entire function. We verify the hypotheses of Lemma~\ref{lem:finite-order-criterion}.
	
	For each \(n\), let
	\[
	q_n\defeq x_n+i(y_n-3\eta_0/2+\delta_n)
	\]
	be the centre of the gate \(I_n(\uldelta)\). Join \(-1\) to \(q_1\) by a fixed curve in $V$, and join \(q_n\) to \(q_{n+1}\) by the straight line segment. This gives a continuous curve \(\Gamma\subset V\) tending to infinity. Since \(\delta_n\in[0,3\eta_0]\), the vertical displacement of \(q_n\) from \(w_n\) lies in \([-3\eta_0/2,3\eta_0/2]\). Choosing \(\eta_0\) sufficiently small, the curve \(\Gamma\) stays a positive Euclidean distance from \(\partial V\); that is, there is \(\alpha>0\) such that \(\dist(\Gamma,\partial V)\ge\alpha\). Moreover, since \(|y_{n+1}-y_n|\le2\pi\) and \(x_{n+1}-x_n\ge K_0\), the Euclidean length of \(\Gamma\cap\{z:\Rea z\le t\}\) is at most \(K_1t\), for all sufficiently large \(t\), with some constant \(K_1>0\).

	For every sufficiently large \(t\), choose \(\zeta_t\in\Gamma\) with \(\Rea\zeta_t=t\). Since \(\rho_V(z)\le2/\dist(z,\partial V)\), we have \(\rho_V(z)\le2/\alpha\) on \(\Gamma\). Hence
	\[
	d_V(-1,\zeta_t)\le \frac{2K_1}{\alpha}t.
	\]
	Thus the first hypothesis of Lemma~\ref{lem:finite-order-criterion} holds.
	
	It remains to verify the second hypothesis. Choose 
	\(B>2\log32\), and put \(C=1\). If \(t\) is sufficiently large, then
	\(t+B>x_1\). For every \(s>x_1\), the vertical slice
	\(V\cap\{\Rea z=s\}\) is connected: if \(s\ne x_n\) for all \(n\), this is
	the full vertical slice of the tube \(U\), while if \(s=x_n\), it is exactly
	the gate \(I_n(\uldelta)\). Hence, for each sufficiently large \(t\), we may
	take \(t'=t+B\in[t+B,t+B+C]\), and the second hypothesis of
	Lemma~\ref{lem:finite-order-criterion} follows.
	
	Both hypotheses of Lemma~\ref{lem:finite-order-criterion} are satisfied,
	and therefore \(f\) has finite order.
\end{proof}

Corollary~\ref{cor:choose_delta} proves the first assertion of Theorem~\ref{th:2}. If \(|z_{n+1}|\geq K|z_n|\) for some \(K>1\), then \eqref{eq_separation} holds with \(K_0=\log K\), and Proposition~\ref{prop:finite_order} gives finite order. Together with the rescaling reduction above, this completes the proof of Theorem~\ref{th:2}.

\section{Proof of Theorem~\ref{th:1}}
\label{S.th1}

We follow a strategy similar to that in the proof of Theorem~\ref{th:2}. This time, the definition of our tracts will be more involved, and will depend on three infinite collections of parameters. In all cases, to simplify the estimates, the tracts will be symmetric with respect to the real axis, and will be based on the half-infinite strip
\[
V' \defeq \{ x + iy : x > -2, |y| < \pi/2 \}.
\]
Our tracts will be constructed by introducing \textit{blockages}, indexed by $n\in \N$, into $V'$. Each blockage consists of a passage, a narrow tunnel, and several chambers separated by gates. (The meaning of the terms `blockage', `gate' and `tunnel' will become apparent soon.) To fully define the tracts we will need three collections of parameters, as follows.

\begin{enumerate}
	\item A sequence \(\ulr\defeq(r_0,r_1,\ldots)\) of positive real numbers. In the first step below we shall impose only \(r_0\ge r_0^*\), with \(r_0^*\) large enough for \eqref{eq:G_2}, and \(r_n\ge1\) for \(n\ge1\); later we will choose the gaps larger, if needed, to ensure finite order. We will ensure that the first blockage starts at real part \(r_0\), and the horizontal distance between the \(n\)-th and \((n+1)\)-st blockages will be \(r_n\).
	\item A sequence \(\ult\defeq(t_1,t_2,\ldots)\) of positive real numbers, each at most \(1/4\). The number \(t_n\) determines the height of the tunnel \(T_n\) in the \(n\)-th blockage.
	\item A family of gate parameters \(\uldelta=(\delta_{n,m})\), indexed by
	\(n\in\N\) and \(1\le m\le n\), with \(0<\delta_{n,m}\le 1/n\). The number
	\(\delta_{n,m}\) determines the common height of the two symmetric gates
	\(G^1_{n,m}\) and \(G^{-1}_{n,m}\). We write
	\[
	\Delta\defeq \prod_{n\in\N}\prod_{m=1}^n (0,1/n]
	\]
	for the corresponding parameter space.
\end{enumerate}
To define $V(\ulr, \ult, \uldelta)$ for sequences as above, for simplicity of notation set $\tau_1 \defeq \pi/2$ and $\tau_2 \defeq \tau_1 - 1$. For each $n$,
let 
$$ R_n \defeq \sum_{i=0}^{n-1} r_i + (n-1).$$

Moreover, we define the tunnel $T_n$ and the passage $P_n$ as 
\begin{align*}
	P_n & \defeq \left[R_n, R_n+\frac{1}{2}\right]\times \left[-\tau_2 i,\tau_2 i\right], \\
	T_n &\defeq \left[R_n+\frac{1}{2}, R_n+1\right]\times \left[-\frac{t_n}{2}i,\frac{t_n}{2}i\right], 
\end{align*}
and for every $1\leq m\leq n$ and $j\in \{-1,1\}$, the chambers $C^{j}_{n,m}$ and gates $G^{j}_{n,m}$ are given by
\begin{align*}
	C^{j}_{n,m} & \defeq \left[R_n, R_n+1\right]\times \left[j\cdot i\left(\tau_2+\frac{m-1}{n}\right), j\cdot i\left(\tau_2+\frac{m}{n}\right)\right], \\
	G^j_{n,m}
	& \defeq
	\left\{R_n+1+iy:
	\left|y-j\left(\tau_2+\frac{2m-1}{2n}\right)\right|<\frac{\delta_{n,m}}2
	\right\},
	\qquad j\in\{-1,1\}.
\end{align*}
Thus \(\delta_{n,m}\) is the common height of the two symmetric gates \(G^1_{n,m}\) and \(G^{-1}_{n,m}\): the value
\(\delta_{n,m}=1/n\) corresponds to a fully open pair of gates, while the limiting value \(\delta_{n,m}=0\), not included
in \(\Delta\), corresponds to a fully closed gate. See Figure \ref{fig:tract_Erdos}. Using a tunnel, rather than another gate, makes subsequent calculations somewhat easier, since it allows us to control distortion via narrow vertical bottlenecks.
We also set 
	$$x_n\defeq R_n+\frac14 \quad \text{ and } \quad x_n'\defeq R_n+\frac32.$$ 
	Thus \(x_n\) lies in the passage \(P_n\), while \(x_n'\) lies in the strip between the \(n\)-th and \((n+1)\)-st blockages whenever \(r_n\ge1\).
In addition, in order to define our tract, it is convenient to introduce notation for some of the boundaries of the passage \(P_n\) and the tunnel \(T_n\): 
\begin{align*}
	L^1_n&\defeq\{R_n+iy:|y|<\tau_2\},\\
	L^2_n&\defeq\{R_n+1/2+iy:|y|<t_n/2\},\\
	L^3_n&\defeq\{R_n+1+iy:|y|<t_n/2\}.
\end{align*}
Finally, we define 
\begin{equation*}
	V(\ulr,\ult,\uldelta)\defeq
	\left(
	V'\setminus
	\bigcup_{n\in\N}
	\left(
	\partial P_n\cup \partial T_n
	\cup
	\bigcup_{\substack{1\le m\le n\\ j\in\{-1,1\}}}
	\partial C^j_{n,m}
	\right)
	\right)
	\cup
	\bigcup_{n\in\N}
	\left(
	\bigcup_{\substack{1\le m\le n\\ j\in\{-1,1\}}} G^j_{n,m}
	\cup
	\bigcup_{k\in\{1,2,3\}} L^k_n
	\right).
\end{equation*}
Thus we remove the individual boundary walls of the passage, tunnel and chambers, and then add back only the gates and the selected vertical openings. 
\begin{figure}[htp]	
	\centering
	\def\svgwidth{\linewidth}
	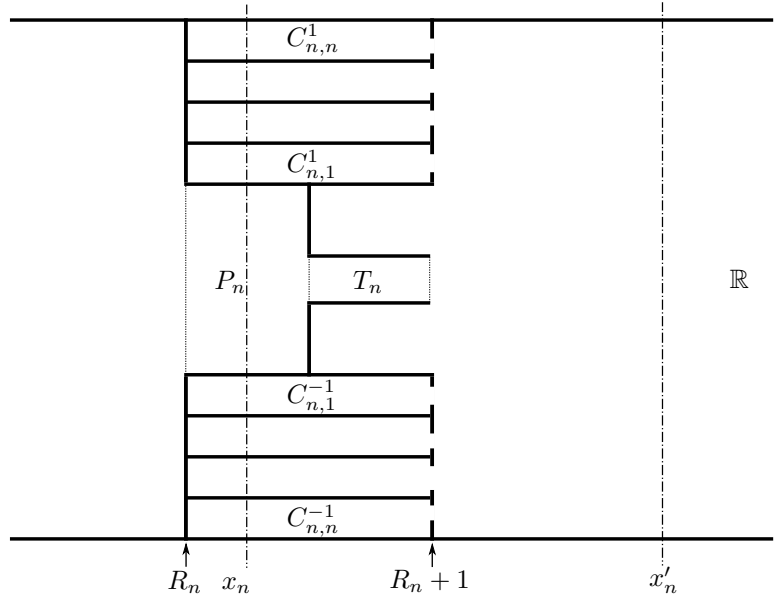
	\caption{Schematic of the tract $V(\ulr,\ult,\uldelta)$. 
		Each blockage (shown for a fixed $n$) consists of a central passage $P_n$, a narrow tunnel $T_n$, 
		and chambers $C^{j}_{n,m}$ above and below, connected through gates $G^{j}_{n,m}$. 
		The vertical segments $L^1_n$, $L^2_n$, and $L^3_n$ indicate parts of the boundary that are retained. 
		The vertical lines at real parts $x_n$ and $x_n'$ mark the locations where the maximum modulus will be controlled.}\label{fig:tract_Erdos}
\end{figure}

Observe that by definition, $V(\ulr, \ult, \uldelta)\subset V'$ is simply connected and disjoint from its $2\pi i$ translates. We let $G_{(\ulr, \ult, \uldelta)}$ be the conformal isomorphism from $V(\ulr, \ult, \uldelta)$ to $H$, such that $G_{(\ulr, \ult, \uldelta)}(-1) = 1$, and such that $G_{(\ulr, \ult, \uldelta)}(z) \rightarrow \infty$ as $\Rea z \rightarrow +\infty$. Let $M > 0$ be the constant from Theorem~\ref{thm_approx1.7}, which does not depend on $(\ulr, \ult, \uldelta)$.
Arguing as in Observation~\ref{claim:ell}, if \(r_0\) is chosen sufficiently large, in particular so that \(r_0>2(\log M+1)\), then for any vertical segment \(\beta\subset V(\ulr,\ult,\uldelta)\) at real part \(x>r_0/2\), with endpoints in \(\partial V(\ulr,\ult,\uldelta)\), which separates \(-1\) from \(\infty\), we have
\begin{equation}
	\label{eq:G_2}
	\sup_{z \in \beta} \Rea G_{(\ulr, \ult, \uldelta)}(z) \geq 4.
\end{equation}
Let $g_{{(\ulr, \ult, \uldelta)}}$ be the corresponding function satisfying \eqref{eq:gG}, corresponding to the model $(G_{(\ulr, \ult, \uldelta)}, V(\ulr, \ult, \uldelta), H)$. We then let $f_{{(\ulr, \ult, \uldelta)}}$ and $h_{{(\ulr, \ult, \uldelta)}}$ be the corresponding functions provided by Theorem \ref{thm_approx1.7}. Since \(V(\ulr,\ult,\uldelta)\) is invariant under complex conjugation and the
normalisations are real, the corresponding map \(G_{(\ulr,\ult,\uldelta)}\) is
real-symmetric. Hence, for \(T=\exp(V(\ulr,\ult,\uldelta))\), the map
\(\Psi\colon T\to H\) satisfies the symmetry hypotheses in
Theorem~\ref{thm_approx1.7}. We therefore choose the approximating function
\(f_{(\ulr,\ult,\uldelta)}\) so that
\(f_{(\ulr,\ult,\uldelta)}(\bar\zeta)=
\overline{f_{(\ulr,\ult,\uldelta)}(\zeta)}\).

Then, for any \(\ulr\) whose first entry \(r_0\) is chosen so that \eqref{eq:G_2} holds, arguing as in Proposition~\ref{prop_MaxInV}, 
\begin{equation}
	\Mlog(f_{_{(\ulr,\ult,\uldelta)}}) \cap \{x+iy\in\C:x>r_0/2\}
	\subset
	\bigcup_{k\in\Z}\bigl(V(\ulr,\ult,\uldelta)+2k\pi i\bigr).
\end{equation}
From now on, we will assume that \(r_0\) is chosen so that \eqref{eq:G_2} holds.

With this notation, the vertical line \(\{\Rea z=x_n\}\) passes through the chambers and the passage of the \(n\)-th blockage. Our goal is to arrange, by varying the tunnel height and the gate sizes, that the maxima of \(z\mapsto |f_{(\ulr,\ult,\uldelta)}(e^z)|\) on the passage and chamber slices of this line have the relative sizes required for the shooting argument.

For \(n\in\N\), the vertical line \(\{\Rea z=x_n\}\) intersects the tract
inside the passage and inside each chamber of the \(n\)-th blockage. We denote
the components of this intersection by
\begin{align*}
	\gamma_n
	&\defeq V(\ulr,\ult,\uldelta)\cap P_n\cap\{\Rea z=x_n\}
	= \{x_n+iy: |y|<\tau_2\},\\
	\gamma^j_{n,m}
	&\defeq V(\ulr,\ult,\uldelta)\cap C^j_{n,m}\cap\{\Rea z=x_n\}
	= \{x_n+iy: \tau_2+(m-1)/n<jy<\tau_2+m/n\},
\end{align*}
for \(1\le m\le n\) and \(j\in\{-1,1\}\). Hence
\[
V(\ulr,\ult,\uldelta)\cap\{\Rea z=x_n\}
=
\gamma_n\cup
\bigcup_{m=1}^n\bigcup_{j\in\{-1,1\}}\gamma^j_{n,m},
\]
as a disjoint union of \(2n+1\) open vertical segments. We also set
\[
\xi_n\defeq x_n,\qquad
\xi^j_{n,m}\defeq x_n+i\,j\left(\tau_2+\frac{m-\frac12}{n}\right).
\]
Thus \(\xi_n\) is the centre of the passage slice \(\gamma_n\), and
\(\xi^j_{n,m}\) is the centre of the corresponding chamber slice
\(\gamma^j_{n,m}\). Compact subsegments of these slices will be chosen later
for the shooting argument.

The next proposition is the open-gate estimate. It shows that, once the
tunnel heights have been chosen sufficiently small, a fully open gate forces
the corresponding chamber slice to dominate the passage slice: the supremum
on \(\gamma^j_{n,m}\) is strictly larger than the supremum on \(\gamma_n\).

\begin{proposition}[Open-gate comparison]\label{prop_fixing_tn}
	There are an initial lower bound \(r_0^*>0\) and tunnel heights \(\ult=(t_1,t_2,\ldots)\), with \(0<t_n\leq1/4\), such that for every sequence \(\ulr=(r_0,r_1,\ldots)\) satisfying \(r_0\ge r_0^*\) and \(r_k\ge1\) for all \(k\ge1\), the following holds. If \(n\in\N\), \(1\leq m\leq n\), \(j\in\{-1,1\}\), and \(\uldelta\in\Delta\) satisfies \(\delta_{n,m}=1/n\), then
	\[
	\sup_{z\in\gamma^j_{n,m}}|f_{(\ulr,\ult,\uldelta)}(e^z)|
	\geq
	|f_{(\ulr,\ult,\uldelta)}(e^{\xi^j_{n,m}})|
	>
	\sup_{w\in\gamma_n}|f_{(\ulr,\ult,\uldelta)}(e^w)|.
	\]
\end{proposition}
\begin{proof}
	Choose \(r_0^*>2(\log M+1)\) so large that \eqref{eq:G_2}
	holds whenever \(r_0\ge r_0^*\). We shall choose the tunnel heights
	\(\ult=(t_1,t_2,\ldots)\) inductively. Throughout this proposition, admissible sequences \(\ulr\) are those satisfying \(r_0\ge r_0^*\) and \(r_k\ge1\) for all \(k\ge1\). Since \(r_n\ge1\), the point \(x_n'\) lies in the strip between the \(n\)-th and \((n+1)\)-st blockages,
	because \(R_{n+1}=R_n+r_n+1\ge R_n+2\). Fix \(n\in\N\), and suppose that \(t_1,\ldots,t_{n-1}\) have already been chosen. The value of \(t_n\in(0,1/4]\) will be fixed at the end of the inductive step.
	
	Let \(S\) be the strip defined in \eqref{eq_S}, and let
	\(\psi:S\to H\) be the normalized conformal isomorphism from
	Lemma~\ref{lem:psi_basic_estimates}. Write  \(V\defeq V(\ulr,\ult,\uldelta)\),
	\(G\defeq G_{(\ulr,\ult,\uldelta)}\), and
	\(\phi\defeq\psi^{-1}\circ G:V\to S\). Since the tract is symmetric with
	respect to the real axis, and since the normalizations are real, both
	\(G\) and \(\phi\) map the real axis to the real axis. Let
	\[
	\beta_n\defeq V\cap\{\Rea z=x_n'\}.
	\]
	The segment \(\beta_n\) is the full vertical crosscut of the clean strip at real part \(x_n'\), and \(\gamma_n\) separates \(-1\) from \(\beta_n\) in \(V\).  For \(x_n\le s\le x_n'\), let \(\theta_n(s)\) be the length of the
	shortest vertical segment at real part \(s\) that separates \(\gamma_n\)
	from \(\beta_n\) in \(V\). On the tunnel interval
	\(R_n+1/2<s<R_n+1\), such a separating segment has length at most \(t_n\).
	Hence
	\[
	\int_{x_n}^{x_n'}\frac{ds}{\theta_n(s)}
	\ge \frac{1}{2t_n}.
	\]	
	Since \(t_n\le1/4\), the hypothesis of Theorem~\ref{thm_ahlfors} is
	satisfied. Therefore
	\[
	\inf_{z\in\beta_n}\Rea\phi(z)
	-
	\sup_{w\in\gamma_n}\Rea\phi(w)
	\ge
	\frac{1}{2t_n}-\frac1\pi\log32.
	\]
	Set
	\[
	a_n\defeq\sup_{w\in\gamma_n}\Rea\phi(w),
	\qquad
	b_n\defeq\phi(x_n').
	\]
	Since \(x_n'\in\beta_n\), we obtain
	$
	b_n-a_n\ge \frac{1}{2t_n}-\frac1\pi\log32.
	$
	In particular \(b_n-a_n\to+\infty\) as \(t_n\to0\), uniformly in the gate
	parameters and in all later choices of parameters. Also \(a_n\ge\Rea\phi(x_n)=\phi(x_n)>0\), and hence \(b_n\to+\infty\)
	uniformly as \(t_n\to0\).

	We next estimate how the tunnel height \(t_n\) controls \(\Rea G\) on the passage slice \(\gamma_n\). Let \(w\in\gamma_n\),
	and write \(\phi(w)=s+iy\), with \(|y|<1/2\). By definition of \(a_n\),
	we have \(s=\Rea\phi(w)\le a_n\). Hence Lemma~\ref{lem:psi_basic_estimates}
	gives
	$
	\Rea G(w)
	=
	\Rea\psi(s+iy)
	\le
	\psi(s)
	\le
	\psi(a_n+1),
	$
	where the last inequality follows from \eqref{eq:lem_aux1}. Therefore
	\[
	\sup_{w\in\gamma_n}\Rea G(w)
	\le
	\psi(a_n+1).
	\]
	Since \(b_n\in\R\), we have \(G(x_n')=\psi(b_n)>0\). Once \(t_n\) is small
	enough that \(b_n>a_n+1\), another application of
	Lemma~\ref{lem:psi_basic_estimates}, with \(s=a_n+1\) and
	\(u=b_n-a_n-1\), gives
	\begin{equation}\label{eq:in41_G}
		\frac{\sup_{w\in\gamma_n}\Rea G(w)}{G(x_n')}
		\le
		\frac{\psi(a_n+1)}{\psi(b_n)}
		\le
		e^{-a_\psi(b_n-a_n-1)}	\le
		e^{
			-a_\psi\left(\frac{1}{2t_n}-\frac{1}{\pi}\log32-1\right)} \xrightarrow[t_n\to0]{}0
	\end{equation}
	where the convergence of the right-hand side to \(0\) is independent of
	the gate parameters and of all later choices in the induction. In other words, by making the tunnel height \(t_n\) sufficiently small, we
	force the real part of \(G\) on the passage \(\gamma_n\) to be arbitrarily
	small compared with the value \(G(x_n')\) at the reference point
	\(x_n'\).
	
	We next compare \(x_n'\) with the chamber centres whose gates are fully open. Suppose that \(\delta_{n,m}=1/n\). Then \(\xi^j_{n,m}\) can be joined to \(x_n'\) by a polygonal curve in \(V\) passing through the centre of the open gate \(G^j_{n,m}\) and then through the strip to the right of the blockage until reaching $x_n'$. The Euclidean length of this curve is bounded above, and its distance from \(\partial V\) is bounded below, by constants depending only on \(n,m,j\). Hence \eqref{eq:hypest} gives a constant \(K^j_{n,m}>0\) such that
	$
	d_V(\xi^j_{n,m},x_n')\le K^j_{n,m}
	$
	whenever \(\delta_{n,m}=1/n\). In order to get uniform constants, let
	\[
	K_n\defeq
	\max_{\substack{1\le m\le n\\ j\in\{-1,1\}}}K^j_{n,m}.
	\]
	Consider the hyperbolic closed ball
	$
	E_n\defeq\{\eta\in S:d_S(0,\eta)\le K_n\}$. If \(\eta=u+iv\in S\), then \(|v|<1/2\), and hence
	$
	\Rea(e^{\pi\eta})=e^{\pi u}\cos(\pi v)>0.
	$
	Since \(E_n\) is compact and contained in \(S\), the continuous function
	\(\eta\mapsto \Rea(e^{\pi\eta})\) has a positive minimum on \(E_n\). Thus
	we may choose \(c_n\in(0,1)\) so small that
	$
	2c_n\le \min_{\eta\in E_n}\Rea e^{\pi\eta}.
	$ Thus, by Corollary~\ref{cor:psi_large_real},
	\[
	\frac{\psi(b+\eta)}{\psi(b)}\longrightarrow e^{\pi\eta}
	\qquad\text{as } b\to+\infty
	\]
	locally uniformly in \(S\). Since \(E_n\) is compactly contained in \(S\),
	this convergence is uniform for \(\eta\in E_n\), and so there exists \(B_n>0\) such that, for every \(b\ge B_n\) and every
	\(\eta\in E_n\),
	\[
	\Rea\psi(b+\eta)\ge c_n\psi(b),
	\]
	where we have used that \(\psi(b)>0\) for \(b\in\R\).
	We now choose the constants that will give a quantitative gap large enough
	to survive the approximation error \(f-g\). Set \(x_n^\flat\) to be the smallest possible value of \(x_n=R_n+1/4\) among all admissible choices of \(\ulr\). Then \(x_n\ge x_n^\flat\) for every admissible \(\ulr\). Choose \(L_n>0\) so large that
	\[
	D_n\defeq e^{c_nL_n}-e^{c_nL_n/2}>2Me^{-x_n^\flat}
	\]
	and choose \(t_n\in(0,1/4]\) sufficiently small so that, for every
	admissible choice of \(\ulr\), of the gate parameters, and of all later
	parameters,
	\[
	b_n\ge B_n,\qquad
	G(x_n')\ge L_n,\qquad
	\sup_{w\in\gamma_n}\Rea G(w)\le \frac{c_n}{2}G(x_n').
	\]
	This is possible by the estimates obtained above: as \(t_n\to0\), we have
	\(b_n\to+\infty\), hence \(G(x_n')=\psi(b_n)\to+\infty\), and also
	$
	\frac{\sup_{w\in\gamma_n}\Rea G(w)}{G(x_n')}\to0.
	$
	This last convergence follows from \eqref{eq:in41_G}.
	These estimates are independent of which gate is open and of all later
	choices in the induction, so the same \(t_n\) works for every possible open
	gate in the \(n\)-th blockage.
	
	Now suppose that \(\delta_{n,m}=1/n\). Since \(\phi:V\to S\) is a
	conformal isomorphism,
	\[
	d_S(\phi(\xi^j_{n,m}),\phi(x_n'))
	=
	d_V(\xi^j_{n,m},x_n')
	\le K_n.
	\]
	As \(\phi(x_n')=b_n\), the point
	$
	\eta\defeq \phi(\xi^j_{n,m})-b_n
	$
	belongs to \(E_n\). Therefore, by the estimates above and by the choice of \(t_n\),
	\[
	\Rea G(\xi^j_{n,m})
	=
	\Rea\psi(b_n+\eta)
	\ge c_n\psi(b_n)
	=
	c_nG(x_n') >
	\sup_{w\in\gamma_n}\Rea G(w).
	\]
	More precisely,
	\[
	|g_{(\ulr,\ult,\uldelta)}(e^{\xi^j_{n,m}})|
	-
	\sup_{w\in\gamma_n}|g_{(\ulr,\ult,\uldelta)}(e^w)|
	\ge
	e^{c_nG(x_n')}-e^{c_nG(x_n')/2}
	\ge D_n.
	\]
	Finally, since
	\(f=f_{(\ulr,\ult,\uldelta)}=g_{(\ulr,\ult,\uldelta)}+h_{(\ulr,\ult,\uldelta)}\), Theorem~\ref{thm_approx1.7}
	gives
	\(
	|h_{(\ulr,\ult,\uldelta)}(e^z)|\le Me^{-x_n}\le Me^{-x_n^\flat}
	\) on \(\{\Rea z=x_n\}\), and so
	\[
	|f_{(\ulr,\ult,\uldelta)}(e^{\xi^j_{n,m}})|
	-
	\sup_{w\in\gamma_n}|f_{(\ulr,\ult,\uldelta)}(e^w)|
	\ge
	D_n-2Me^{-x_n^\flat}>0.
	\]
	Since \(\xi^j_{n,m}\in\gamma^j_{n,m}\), it follows that
	\[
	\sup_{z\in\gamma^j_{n,m}}
	|f_{(\ulr,\ult,\uldelta)}(e^z)|
	\ge
	|f_{(\ulr,\ult,\uldelta)}(e^{\xi^j_{n,m}})|
	>
	\sup_{w\in\gamma_n}
	|f_{(\ulr,\ult,\uldelta)}(e^w)|.
	\]
	This completes the \(n\)-th inductive step. Since the estimates are uniform in the remaining gate parameters and in all later choices of gaps and tunnel heights, the induction gives the desired initial lower bound \(r_0^*\) and sequence \(\ult\).
\end{proof}

\begin{proposition}\label{prop_fixing_rn}
	Let \(r_0^*>0\) and \(\ult=(t_1,t_2,\ldots)\) be the initial lower bound and tunnel heights provided by Proposition~\ref{prop_fixing_tn}. Then there exists a sequence \(\ulr=(r_0,r_1,\ldots)\), with \(r_0\ge r_0^*\) and \(r_k\ge1\) for all \(k\ge1\), such that, for every \(\uldelta\in\Delta\), the function \(f_{(\ulr,\ult,\uldelta)}\) has finite order.
	
	Moreover, for every \(T>0\), there is a constant \(L_T>0\) such that
	\[
	\Rea G_{(\ulr,\ult,\uldelta)}(z)\le L_T
	\]
	for every \(\uldelta\in\Delta\) and every
	\(z\in V(\ulr,\ult,\uldelta)\) with \(\Rea z\le T\).
\end{proposition}
\begin{proof}
We choose the sequence \(\ulr=(r_0,r_1,\ldots)\) inductively as follows. First choose \(r_0\ge r_0^*\) so large that \(R_1=r_0\ge20/t_1\). If \(n\ge2\) and \(r_0,\ldots,r_{n-2}\) have been chosen, choose \(r_{n-1}\ge5\) so large that \(R_n\ge 20\sum_{k=1}^n 1/t_k\). This is possible because \(R_n\) depends on \(r_0,\ldots,r_{n-1}\). Fix \(\uldelta\in\Delta\). The following estimates are independent of this choice.
	
	We want to apply Lemma~\ref{lem:finite-order-criterion}. Choose a constant \(B>2\log32\). Since the intervals $[R_n,R_n+1]$ are pairwise disjoint and have length $1$, we may choose, say, $C=3$, so that for every sufficiently large \(t\) there is \(t'\in[t+B,t+B+C]\) outside all these intervals. For this choice of \(t'\), the slice \(V\cap\{\Rea z=t'\}\) is the full vertical crosscut of the strip, hence is connected, and it separates the part of \(V\) with real part at most \(t\) from the positive right end of the tract. Thus the second hypothesis of Lemma~\ref{lem:finite-order-criterion} holds.
	
	For the first hypothesis, suppose that $t$ is sufficiently large; in particular, assume that $t>\max\{7,R_1-5\}$.We estimate \(d_V(-1,t)\) by estimating the hyperbolic length of the line segment \([-1,t]\). Let \(N\in\N\) be the greatest value such that \(R_N<t+5\). Observe that if \(\tau\in[R_n-2,R_n+3]\), for some \(n\), then the point \(\tau\) is a distance at least \(t_n/2\) from \(\partial V\). All other real values greater than \(-1\) are a distance at least \(1\) from \(\partial V\). Hence, by the standard estimate of the hyperbolic metric, \eqref{eq:hypest},
	\[
	d_V(-1,t)
	\leq 2(t+1)+5\sum_{k=1}^N\frac4{t_k}
	\leq 2(t+1)+R_N
	<3t+7.
	\]
	In particular, for \(t\) sufficiently large, \(d_V(-1,t)\le4t\), so the first hypothesis of Lemma~\ref{lem:finite-order-criterion} also holds, with $A=4$, again uniformly in \(\uldelta\).
	
	Therefore Lemma~\ref{lem:finite-order-criterion} implies that
	\(f_{(\ulr,\ult,\uldelta)}\) has finite order for every \(\uldelta\in\Delta\).
	Its final assertion also gives the following: for every \(T>0\), there is a
	constant \(L_T>0\), depending only on \(T\), \(H\), and the constants just fixed,
	such that
	\[
	\Rea G_{(\ulr,\ult,\uldelta)}(z)\le L_T
	\]
	whenever \(\uldelta\in\Delta\), \(z\in V(\ulr,\ult,\uldelta)\), and
	\(\Rea z\le T\). Since all constants in the verification above are independent
	of \(\uldelta\), so is \(L_T\). This proves both assertions.
\end{proof}

The next proposition is the almost closed-gate estimate. It is the counterpart to
Proposition~\ref{prop_fixing_tn}: after \(\ulr\) and \(\ult\) have been fixed,
making the common height \(\delta_{n,m}\) of the two symmetric gates
\(G^1_{n,m}\) and \(G^{-1}_{n,m}\) sufficiently small forces both corresponding chamber slices to be dominated by the passage slice. 

\begin{proposition}[Almost closed-gate comparison]\label{prop_closing_deltan}
	Let \(\ulr\) be the sequence given by Proposition~\ref{prop_fixing_rn}, and let \(\ult\) be the sequence given by Proposition~\ref{prop_fixing_tn}. Then there exists a family \((\varepsilon_{n,m})_{n\in\N,\,1\le m\le n}\), with \(\varepsilon_{n,m}\in(0,1/n)\), such that the following holds for every \(n\in\N\) and \(1\le m\le n\). If \(\uldelta\in \Delta\) satisfies \(0<\delta_{n,m}\leq\varepsilon_{n,m}\), then, for both \(j\in\{-1,1\}\),
	\[
	\sup_{z\in\gamma^j_{n,m}} |f_{(\ulr,\ult,\uldelta)}(e^z)|
	<
	\sup_{w\in\gamma_n}|f_{(\ulr,\ult,\uldelta)}(e^w)|.
	\]
	More precisely, the family \((\varepsilon_{n,m})\) may be chosen so that, for every \(n\in\N\) and \(1\le m\le n\), under the same hypothesis, for both \(j\in\{-1,1\}\),
	\begin{equation}\label{eq:prop4.3}
		\sup_{z\in\gamma^j_{n,m}}\Rea G_{(\ulr,\ult,\uldelta)}(z)\leq1.
	\end{equation}
\end{proposition}

\begin{proof}
		Fix \(n\in\N\) and \(1\leq m\leq n\). Throughout the proof, \(\ulr\) and \(\ult\) are fixed, and all gate parameters other than \(\delta_{n,m}\) are arbitrary. First, the segment \(\gamma_n\) separates \(-1\) from the positive right end of the tract. Hence \eqref{eq:G_2} and the identity \(|g(e^w)|=\exp(\Rea G(w))\) give
		\[
		\sup_{w\in\gamma_n}|g_{(\ulr,\ult,\uldelta)}(e^w)|\geq e^4.
		\]
		We next prove the more precise estimate \eqref{eq:prop4.3} for the chamber slices. Once this is established, the claimed comparison for \(f\) will follow from the preceding lower bound and the approximation \(f=g+h\). By the real symmetry established above, it suffices to consider the upper chamber, so we fix \(j=1\). Denote
	$a\defeq \tau_2+\frac{m-1}{n},$
	$b\defeq \tau_2+\frac{m}{n},$ and
	$y_0\defeq \tau_2+\frac{2m-1}{2n},$
	so that
	\[
	C\defeq \operatorname{int} C^1_{n,m}
	=\{x+iy:R_n<x<R_n+1,\ a<y<b\}.
	\]
	Let \(\rho\defeq1/(4n)\). For \(0<s\le\rho\), set
	\[
	I_s\defeq\{R_n+1+iy: |y-y_0|<s/2\}.
	\]
	Thus \(I_s\) is the gate \(G^1_{n,m}\) when \(\delta_{n,m}=s\), and, because
	\(s\le\rho\), it stays a definite distance from the two horizontal sides of
	\(C\).
	
	Fix all parameters \(\delta_{k,l}\), \((k,l)\ne(n,m)\). For \(0<s\le\rho\), let
	\(\uldelta(s)\) be obtained from these fixed parameters by setting
	\(\delta_{n,m}=s\), and write
	\[
	V_s\defeq V(\ulr,\ult,\uldelta(s)),\qquad
	G_s\defeq G_{(\ulr,\ult,\uldelta(s))},\qquad
	u_s\defeq \Rea G_s.
	\]
	We now work only in the fixed rectangle \(C\). As the real part of a holomorphic function, the function \(u_s\) is harmonic in \(C\subset V_s\). By the final assertion of Proposition~\ref{prop_fixing_rn}, applied with \(T=R_n+1\), there is a constant \(B_n>0\), independent of \(s\) and of the remaining gate parameters, such that \(u_s\leq B_n\) on \(I_s\). On the remainder of the boundary the relevant boundary values are non-positive: more precisely, for every \(\zeta\in\partial C\setminus\overline{I_s}\), we have \(\limsup_{C\ni z\to\zeta}u_s(z)\leq0\), since these points lie on \(\partial V_s\) and \(G_s(V_s)=H\), where \(\partial H\subset\{\Rea z\leq0\}\). The two endpoints of \(I_s\) are irrelevant for harmonic measure.
	
	Let
	$
	\omega_s(z)\defeq \omega(z,I_s,C)
	$
	be the harmonic measure of \(I_s\) in \(C\). Equivalently, \(\omega_s\) is the
	harmonic function in \(C\) with boundary values \(1\) on \(I_s\) and \(0\) on
	\(\partial C\setminus I_s\), except at the two endpoints of \(I_s\), which are
	irrelevant for harmonic measure. Applying the maximum principle to
	\(u_s-B_{n}\omega_s\), we obtain
	$
	u_s(z)\le B_{n}\omega_s(z) 
	$
	for $z\in C$.
	Set
	\[
	E_{n,m}\defeq
	\{R_{n}+1+iy: |y-y_0|\le \rho/2\}.
	\]
	Then \(I_s\subset E_{n,m}\) for every \(0<s\le\rho\), and \(E_{n,m}\) is a
	compact subinterval of the right side of \(C\), disjoint from the two corners.
	The point of the following estimate is that, as viewed from the fixed slice \(\gamma^1_{n,m}\), a boundary interval of length \(s\) has harmonic measure \(O(s)\). Let \(P_C(z,\zeta)\) be the Poisson kernel of \(C\), that is, the density of harmonic measure with respect to arclength. Since \(E_{n,m}\) stays away from the corners and from \(\gamma^1_{n,m}\), the kernel \(P_C\) is bounded for \(z\in\gamma^1_{n,m}\) and \(\zeta\in E_{n,m}\). Hence there is a constant \(K_{n,m}>0\) such that
	$
	P_C(z,\zeta)\le K_{n,m}$ for $z\in\gamma^{1}_{n,m}$ and $\zeta\in E_{n,m}.$
	Therefore, for \(z\in\gamma^{1}_{n,m}\),
	\[
	\omega_s(z)
	=
	\int_{I_s}P_C(z,\zeta)\,|d\zeta|
	\le K_{n,m}|I_s|
	=
	K_{n,m}s.
	\]
	Thus
	\[
	\sup_{z\in\gamma^{1}_{n,m}}u_s(z)
	\le B_{n}K_{n,m}s
	\qquad (0<s\le\rho).
	\]
	
	Choose
	$
	\varepsilon_{n,m}\in
	\left(0,\min\left\{\rho,\frac1{B_nK_{n,m}}\right\}\right).
	$
	Then, whenever \(0<\delta_{n,m}=s\leq\varepsilon_{n,m}\), \eqref{eq:prop4.3} holds for \(j=1\). By the symmetry reduction mentioned at the beginning of the proof, it therefore holds for \(j=-1\) as well. In particular, since \(g_{(\ulr,\ult,\uldelta)}\circ\exp
	=
	\exp\circ G_{(\ulr,\ult,\uldelta)}\), we have
	\[
	\sup_{z\in\gamma^{j}_{n,m}}
	|g_{(\ulr,\ult,\uldelta)}(e^z)|\le e
	\qquad (j\in\{-1,1\}).
	\]
	
	It remains only to transfer these estimates from the model function \(g\) to the approximating entire function \(f=g+h\). Since both
	\(\gamma_n\) and \(\gamma^j_{n,m}\) are contained in the line
	\(\{\Rea z=x_n\}\), Theorem~\ref{thm_approx1.7} gives $|h_{(\ulr,\ult,\uldelta)}(e^z)|\le Me^{-x_n}$
	for $\Rea z=x_n$. We assume, by increasing the initial lower bound for \(r_0\) if necessary, that
	\(Me^{-x_n}<e^{-1}\) for every \(n\). Using
	\(f_{(\ulr,\ult,\uldelta)}
	=
	g_{(\ulr,\ult,\uldelta)}
	+
	h_{(\ulr,\ult,\uldelta)}\), we obtain, for each \(j\in\{-1,1\}\),
	\[
	\begin{aligned}
		&\sup_{w\in\gamma_n}|f_{(\ulr,\ult,\uldelta)}(e^w)|
		-
		\sup_{z\in\gamma^j_{n,m}}|f_{(\ulr,\ult,\uldelta)}(e^z)|  \\
		&\quad\ge
		\sup_{w\in\gamma_n}|g_{(\ulr,\ult,\uldelta)}(e^w)|
		-
		\sup_{z\in\gamma^j_{n,m}}|g_{(\ulr,\ult,\uldelta)}(e^z)|
		-2Me^{-x_n}  \\
		&\quad\ge e^4-e-2Me^{-x_n}\ge e^4-e-2e^{-1}>0.
	\end{aligned}
	\]
	This proves the desired comparison.
\end{proof}

Set
\[
\Delta_*\defeq
\prod_{n\in\N}\prod_{m=1}^n
[\varepsilon_{n,m},1/n].
\]
In the shooting argument, we will compare the largest value of
	\(\lvert f_{(\ulr,\ult,\uldelta)}(e^z)\rvert\) on each chamber slice with
	the largest value on the corresponding passage slice. Since these slices
	are open, the largest values need not be attained. We therefore replace
	the slices by fixed compact subsegments. On these subsegments the maxima
	are attained and, as shown below, depend continuously on \(\uldelta\). This will allow us to apply the shooting lemma, Lemma~\ref{lem:shooting}, in the proof of Corollary \ref{cor:full_shooting}. The next proposition chooses these subsegments so that the deleted endpieces are uniformly too small to contain any maximum relevant to the shooting argument.

\begin{proposition}[Choice of starred segments]\label{prop:starred_segments}
	For each \(n\in\N\), there exist compact subsegments
	\(\gamma_n^*\Subset\gamma_n\), symmetric with respect to the real axis, and compact subsegments \((\gamma^1_{n,m})^*\Subset\gamma^1_{n,m}\), \(1\le m\le n\), with \(\xi^1_{n,m}\in(\gamma^1_{n,m})^*\), such that, if \((\gamma^{-1}_{n,m})^*\defeq\{\overline z:z\in(\gamma^1_{n,m})^*\}\), then, for every \(\uldelta\in\Delta_*\),
	\begin{equation}\label{eq:starred_segments1}
		\max_{w\in\gamma_n^*}|f_{(\ulr,\ult,\uldelta)}(e^w)|\ge e^3-e^{-1}.
	\end{equation}
	
	Moreover, if \(\sigma\) is one of the slices \(\gamma_n\) or \(\gamma^j_{n,m}\), and \(\sigma^*\) denotes the corresponding starred subsegment, then for any $\uldelta\in\Delta_*$, 
	\[
	\sup_{z\in\sigma\setminus\sigma^*}|f_{(\ulr,\ult,\uldelta)}(e^z)|\le e^2+e^{-1}.
	\]
\end{proposition}
\begin{proof}
		Fix \(n\in\N\), and write \(G_{\uldelta}=G_{(\ulr,\ult,\uldelta)}\), \(f_{\uldelta}=f_{(\ulr,\ult,\uldelta)}\), \(g_{\uldelta}=g_{(\ulr,\ult,\uldelta)}\) and \(h_{\uldelta}=h_{(\ulr,\ult,\uldelta)}\). We first construct the starred subsegments of \(\gamma_n\) and of the upper chamber slices; the lower ones will then be obtained by reflection. The key local observation is that \(\Rea G_{\uldelta}\) is uniformly small near an endpoint of any of these slices.
		
				Let \(\sigma\) be one of the slices \(\gamma_n\) or \(\gamma^1_{n,m}\), \(1\leq m\leq n\), and let \(p\in\overline{\sigma}\setminus\sigma\) be one of its endpoints. Since \(x_n=R_n+1/4\), the point \(p\) lies in the relative interior of
					a straight boundary wall, away from all corners and variable gates. We may
					therefore choose a small open half-disc \(D\), independent of \(\uldelta\),
					whose straight boundary has midpoint \(p\), such that \(D\) lies in the
					interior of the passage or chamber containing \(\sigma\), and
					\(D\cap\sigma\) is a terminal subinterval of \(\sigma\). Let \(Q\) and \(B\)
					denote, respectively, the open semicircular and straight parts of
					\(\partial D\). We choose \(D\) so small that \(\Rea z<R_n+1\) on
					\(\overline D\). Then
					\[
					Q\subset V(\ulr,\ult,\uldelta)
					\quad\text{and}\quad
					B\subset\partial V(\ulr,\ult,\uldelta)
					\]
					for every \(\uldelta\in\Delta_*\).

			By the final assertion of Proposition~\ref{prop_fixing_rn}, applied with
		\(T=R_n+1\), there exists a constant \(L_n\) such that
		$
		\Rea G_{\uldelta}\le L_n$ on $Q$,
		for every \(\uldelta\in\Delta_*\). For every \(\zeta\in B\) and every \(\uldelta\in\Delta_*\), we have
	$
	\limsup_{D\ni z\to\zeta}\Rea G_{\uldelta}(z)\leq0,
	$
	since \(G_{\uldelta}\) maps \(V(\ulr,\ult,\uldelta)\) conformally onto
	\(H\) and \(\partial H\subset\{\Rea z\leq0\}\). The two points of
	\(\partial D\setminus(Q\cup B)\) are irrelevant for harmonic measure. Hence, by the maximum
			principle in \(D\),
			\[
			\Rea G_{\uldelta}(z)
			\leq
			L_n\,\omega(z,Q,D)
			\qquad
			(z\in D,\ \uldelta\in\Delta_*).
			\]
			Since \(\omega(z,Q,D)\to0\) as \(z\to p\) along \(\sigma\), there is
				a fixed terminal subinterval of \(\sigma\), adjacent to \(p\), on which
				\[
				\Rea G_{\uldelta}(z)\leq2
				\qquad (\uldelta\in\Delta_*).\]
	Apply this construction at both endpoints of every slice under
		consideration, shortening the terminal subintervals further if necessary.
		For \(\gamma_n\), choose the two deleted endpieces symmetrically and short
		enough that a non-degenerate compact segment remains. For each
		\(\gamma^1_{n,m}\), choose them short enough that the remaining compact
		subsegment contains \(\xi^1_{n,m}\), and define the lower starred
		subsegment by reflection. By the real symmetry of the tract and of
		\(G_{\uldelta}\), the same estimate holds on the reflected lower slices.
		Thus, for every slice \(\sigma\), $\Rea G_{\uldelta}(z)\leq2$ for $z\in\sigma\setminus\sigma^*$ and $\uldelta\in\Delta_*$.
	
	All these slices lie on the line \(\{\Rea z=x_n\}\).  Since \(g_{\uldelta}\circ\exp=\exp\circ G_{\uldelta}\), \(f_{\uldelta}=g_{\uldelta}+h_{\uldelta}\), and our standing choice of \(r_0\) gives \(|h_{\uldelta}(e^z)|\le Me^{-x_n}<e^{-1}\) on this line, we obtain \(|f_{\uldelta}(e^z)|\le e^2+e^{-1}\) for \(z\in\sigma\setminus\sigma^*\), uniformly in \(\uldelta\in\Delta_*\).
	
	It remains to prove the lower bound on \(\gamma_n^*\).  Fix \(\uldelta\in\Delta_*\).  Since \(\gamma_n\) separates \(-1\) from the positive right end of the tract, \eqref{eq:G_2} gives \(\sup_{w\in\gamma_n}\Rea G_{\uldelta}(w)\ge4\).  Hence there is \(w\in\gamma_n\) with \(\Rea G_{\uldelta}(w)>3\).  This point cannot lie in the deleted endpieces, where \(\Rea G_{\uldelta}\le2\), and so \(w\in\gamma_n^*\).  Therefore \(\max_{\zeta\in\gamma_n^*}|f_{\uldelta}(e^\zeta)|\ge |f_{\uldelta}(e^w)|\ge e^{\Rea G_{\uldelta}(w)}-e^{-1}>e^3-e^{-1}\).  This proves the proposition.
\end{proof}

\begin{lemma}
	\label{lem:finite_shooting}
	For each \(n\in\N\) and \(1\leq m\leq n\), define
	\[
	A_{n,m}(\uldelta)
	\defeq
	\max_{z\in(\gamma^1_{n,m})^*}|f_{(\ulr,\ult,\uldelta)}(e^z)|
	=
	\max_{z\in(\gamma^{-1}_{n,m})^*}|f_{(\ulr,\ult,\uldelta)}(e^z)|,
	\]
	where the equality follows from the real symmetry of \(f_{(\ulr,\ult,\uldelta)}\), and set
	\[
	\phi_{n,m}\colon\Delta_*\to\R,
	\qquad
	\phi_{n,m}(\uldelta)\defeq
	A_{n,m}(\uldelta)
	-
	\max_{w\in\gamma_n^*}|f_{(\ulr,\ult,\uldelta)}(e^w)|.
	\]
	Then:
	\begin{enumerate}
		\item The function \(\phi_{n,m}\) is well-defined and continuous on \(\Delta_*\).
		\item If \(\delta_{n,m}=1/n\), then \(\phi_{n,m}(\uldelta)>0\).
		\item If \(\delta_{n,m}=\varepsilon_{n,m}\), then \(\phi_{n,m}(\uldelta)<0\).
	\end{enumerate}
\end{lemma}
\begin{proof}
	Let \(\sigma\) be one of the fixed compact segments \((\gamma^1_{n,m})^*\), \((\gamma^{-1}_{n,m})^*\) or \(\gamma_n^*\).  For each \(\uldelta\in\Delta_*\), the function \(z\mapsto |f_{(\ulr,\ult,\uldelta)}(e^z)|\) is continuous on \(\sigma\), so the corresponding maximum is attained.  Moreover, if \(\uldelta'\to\uldelta\) in \(\Delta_*\), then Observation~\ref{obs_continuous_dependence} gives locally uniform convergence of \(f_{(\ulr,\ult,\uldelta')}\) to \(f_{(\ulr,\ult,\uldelta)}\) on the compact set \(\exp(\sigma)\).  Hence
	\[
	\left|\max_{z\in\sigma}|f_{(\ulr,\ult,\uldelta')}(e^z)|-\max_{z\in\sigma}|f_{(\ulr,\ult,\uldelta)}(e^z)|\right|
	\le
	\sup_{z\in\sigma}|f_{(\ulr,\ult,\uldelta')}(e^z)-f_{(\ulr,\ult,\uldelta)}(e^z)|\to0.
	\]
	The equality between the two chamber maxima follows from the real symmetry of \(f_{(\ulr,\ult,\uldelta)}\) and from the definition of the lower starred segment by reflection.  Thus \(A_{n,m}\) and \(\phi_{n,m}\) are well-defined and continuous.
	
	Suppose first that \(\delta_{n,m}=1/n\).  Proposition~\ref{prop:starred_segments} gives \(\xi^1_{n,m}\in(\gamma^1_{n,m})^*\), and Proposition~\ref{prop_fixing_tn} gives
	\[
	A_{n,m}(\uldelta)
	\ge |f_{(\ulr,\ult,\uldelta)}(e^{\xi^1_{n,m}})|
	>
	\sup_{w\in\gamma_n}|f_{(\ulr,\ult,\uldelta)}(e^w)|
	\ge
	\max_{w\in\gamma_n^*}|f_{(\ulr,\ult,\uldelta)}(e^w)|.
	\]
	Therefore \(\phi_{n,m}(\uldelta)>0\).

	Suppose now that \(\delta_{n,m}=\varepsilon_{n,m}\).   By equation \eqref{eq:prop4.3} in Proposition~\ref{prop_closing_deltan}, \(|g_{(\ulr,\ult,\uldelta)}(e^z)|\le e\) on \(\gamma^j_{n,m}\), \(j=\pm1\).  Since these segments are contained in the
	line \(\{\Rea z=x_n\}\), Theorem~\ref{thm_approx1.7}, applied with
	\(\zeta=e^z\), gives
	\[
	|h_{(\ulr,\ult,\uldelta)}(e^z)|
	\le \frac{M}{|e^z|}
	=
	Me^{-x_n}.
	\]
	By our standing choice of \(r_0\), made after fixing the universal constant
	\(M\) from Theorem~\ref{thm_approx1.7}, we have \(Me^{-x_n}<e^{-1}\) for every
	\(n\). Therefore
	$
	A_{n,m}(\uldelta)\le e+e^{-1}.
	$
	On the other hand, \eqref{eq:starred_segments1} in Proposition~\ref{prop:starred_segments} says that
	$
	\max_{w\in\gamma_n^*}
	|f_{(\ulr,\ult,\uldelta)}(e^w)|
	\ge e^3-e^{-1}.
	$
	Since \(e+e^{-1}<e^3-e^{-1}\), we obtain \(\phi_{n,m}(\uldelta)<0\).
\end{proof}
\begin{corollary}
	\label{cor:full_shooting}
	There exists \(\uldelta\in\Delta_*\) such that
	\begin{equation}\label{eq:final_shooting}
		\phi_{n,m}(\uldelta)=0
		\qquad\text{for every } n\in\N,\ 1\le m\le n.
	\end{equation}
	In particular, $v_{f_{(\ulr,\ult,\uldelta)}}(e^{x_n})\ge 2n+1$ for every  $n\in\N.$
\end{corollary}

\begin{proof}
	Relabel the pairs \((n,m)\), \(1\le m\le n\), as a single sequence indexed by \(k\in\N\).  If \(k\) corresponds to \((n,m)\), put \(a_k\defeq\varepsilon_{n,m}\), \(b_k\defeq1/n\), \(\delta_k\defeq\delta_{n,m}\) and \(\Phi_k\defeq-\phi_{n,m}\).  By Lemma~\ref{lem:finite_shooting}, \(\Phi_k\ge0\) on the face \(\delta_k=a_k\) and \(\Phi_k\le0\) on the face \(\delta_k=b_k\).  Lemma~\ref{lem:shooting}, applied to \(\Delta_*\) and to the functions \(\Phi_k\), gives \(\uldelta\in\Delta_*\) with \(\Phi_k(\uldelta)=0\) for every \(k\), which is \eqref{eq:final_shooting}.
	
	Fix \(n\in\N\), and write \(f\defeq f_{(\ulr,\ult,\uldelta)}\) and
	\(V\defeq V(\ulr,\ult,\uldelta)\). Let
	\[
	\mathcal S_n\defeq
	\{\gamma_n\}\cup
	\{\gamma^j_{n,m}:1\le m\le n,\ j\in\{-1,1\}\}
	\]
	and recall that 
	$
	V\cap\{\Rea z=x_n\}
	=\bigcup_{\sigma\in\mathcal S_n}\sigma.$
	For
	\(\sigma\in\mathcal S_n\), let \(\sigma^*\) denote the corresponding starred
	subsegment. Define
	$
	\Lambda_n\defeq
	\max_{z\in\gamma_n^*}|f(e^z)|
	$ and note that 
	by \eqref{eq:final_shooting},
	\[
	\max_{z\in\sigma^*}|f(e^z)|=\Lambda_n
	\qquad\text{for every }\sigma\in\mathcal S_n.
	\]
	Moreover, Proposition~\ref{prop:starred_segments} gives
	\(\Lambda_n\ge e^3-e^{-1}\), and also gives
	\[
	\sup_{z\in\sigma\setminus\sigma^*}|f(e^z)|
	\le e^2+e^{-1}
	\qquad\text{for every }\sigma\in\mathcal S_n.
	\]
	Since \(e^2+e^{-1}<e^3-e^{-1}\le\Lambda_n\), the maximum on each full slice
	\(\sigma\in\mathcal S_n\) is attained on the starred subsegment \(\sigma^*\),
	and this maximum is equal to \(\Lambda_n\).
	
	It remains to compare these values with the rest of the vertical line
	\(\{\Rea z=x_n\}\). If \(z\) does not belong to any translate
	\(V+2k\pi i\), then \(f(e^z)=h_{(\ulr,\ult,\uldelta)}(e^z)\). By the
	approximation estimate and the standing choice of \(r_0\),
	$
	|f(e^z)|=|h_{(\ulr,\ult,\uldelta)}(e^z)|<e^{-1}<\Lambda_n.
	$
	If instead \(z\in V+2k\pi i\) for some \(k\in\Z\), then \(z-2k\pi i\in V\), and
	$
	|f(e^z)|=|f(e^{z-2k\pi i})|,
	$
	so the preceding analysis of the slices in \(V\) applies. Thus
	$
	v_{f_{(\ulr,\ult,\uldelta)}}(e^{x_n})\ge 2n+1.
	$
\end{proof}

\begin{proof}[Proof of Theorem~\ref{th:1}]
	Let \(\ulr\), \(\ult\), and \(\uldelta\in\Delta_*\) be the parameters obtained above, and set \(f\defeq f_{(\ulr,\ult,\uldelta)}\). By Theorem~\ref{thm_approx1.7}, \(f\in\mathcal B\). By Proposition~\ref{prop_fixing_rn}, \(f\) has finite order. Finally, by Corollary~\ref{cor:full_shooting}, \(v_f(e^{x_n})\ge 2n+1\) for every \(n\in\N\). Since \(x_n\to+\infty\), it follows that \(\limsup_{r\to\infty}v_f(r)=\infty\).
\end{proof}

\bibliographystyle{alpha}
\bibliography{References}
\end{document}